\documentclass[hidelinks,11pt]{article}
\usepackage[toc,page]{appendix}

\author{Luke Peachey}
\title{Non-uniqueness of curve shortening flow}
\date{}

\usepackage[backend=biber,style=alphabetic]{biblatex}
\usepackage{amsfonts}
\usepackage{amsmath}
\usepackage{amssymb}
\usepackage{amsthm}
\usepackage{appendix}
\usepackage{array}
\usepackage[UKenglish]{babel}
\usepackage{breqn}
\usepackage{csquotes}
\usepackage{enumitem}
\usepackage{faktor}
\usepackage{framed}
\usepackage{fullpage}
\usepackage{graphicx}
\usepackage{mathrsfs}
\usepackage{parskip}
\usepackage{setspace}
\usepackage{subcaption}
\usepackage{thmtools, thm-restate}
\usepackage{tikz}
\usepackage{tikz-cd}

\usepackage{hyperref}
\hypersetup{colorlinks=true,
linkcolor=blue,
citecolor=red,
pdftitle={Non-uniqueness of curve shortening flow v1.0},
}
\usepackage[margin=1.2in]{geometry}

\graphicspath{{./Images/}}

\newcommand{\abs}[1]{\lvert#1\rvert}
\newcommand{\norm}[1]{\lVert#1\rVert}

\newtheorem{thm}{Theorem}[section]
\newtheorem{lem}[thm]{Lemma} 
\newtheorem{prop}[thm]{Proposition}
\newtheorem*{claim}{Claim} 
\newtheorem{cor}[thm]{Corollary}
\newtheorem{conj}[thm]{Conjecture}

\theoremstyle{definition}
\newtheorem{defn}[thm]{Definition}
\newtheorem{eg}[thm]{Example}

\theoremstyle{remark} 
\newtheorem{rem}[thm]{Remark}

\numberwithin{equation}{section}

\DeclareMathOperator\Ima{Im}

\def\XXint#1#2#3{{\setbox0=\hbox{$#1{#2#3}{\int}$ }
\vcenter{\hbox{$#2#3$ }}\kern-.6\wd0}}

\begin{document}

\maketitle

\begin{abstract}
We formulate a uniqueness conjecture for curve shortening flow of proper curves on certain symmetric surfaces and give an example of a non-flat metric on the plane with respect to which curve shortening flow is not unique. That is, with respect to a suitably chosen metric, we construct a non-static solution to curve shortening flow starting from a properly embedded geodesic.
\end{abstract}

\tableofcontents

\section{Introduction}

Given a complete Riemannian surface $(M^2,g)$ and a smooth map $\gamma: \mathbb{R}\times(0,T) \rightarrow M$ such that $\gamma(\cdot,t)$ is a smoothly embedded curve at each time $t \in (0,T)$, we say that the family of curves $\gamma(\cdot,t)$ evolves under curve shortening flow (CSF) if
\begin{equation}\label{eqn:tang}
\left< \partial_t \gamma , \nu \right>_g = \kappa \quad \textnormal{on} \quad \mathbb{R} \times (0,T),
\end{equation}
where $\nu$ is a choice of unit normal vector to the curve, $\kappa := \left< \nabla_{\tau} \tau , \nu \right>_g$ is the geodesic curvature of the curve with respect to $\nu$, and $\tau$ is the unit tangent vector to the curve.\par
If there exists a continuous extension $\gamma:\mathbb{R} \times [0,T) \rightarrow M$ of our map, we say that $\gamma(\cdot,t)$ is a solution to CSF with initial data $\gamma(\cdot,0)$.\par
Given some smooth, properly embedded curve $\gamma(\cdot,0)$ in $M$, it is natural to ask whether solutions to (\ref{eqn:tang}) starting from this initial data  exist and if they are unique within a particular class of functions. In the case of closed curves $\eta:S^1 \times [0,T) \rightarrow M$, existence and uniqueness in the space of smooth solutions to curve shortening flow, and more generally to mean curvature flow, follows from the equation being quasi-parabolic and the maximum principle \cite{gage1986heat}. In fact, depending on how you require your initial data to be achieved, existence and uniqueness can be shown even after dropping the regularity of the initial data to a finite length Jordan curve \cite{lauer2013new}. Despite these results in the closed case, similar fundamental questions regarding existence and uniqueness in the non-closed case remain open.\par

If our ambient space is the flat plane and our initial data is a smooth properly embedded curve which cuts the plane into two regions, each having infinite area, then there exists an immortal solution to CSF $\gamma:\mathbb{R} \times [0,\infty) \rightarrow \mathbb{R}^2$ starting from this initial data \cite{chou1998shortening}. An existence result for proper curves in general ambient spaces remains open.\par

Before discussing uniqueness, we introduce the following class of solutions.

\begin{defn}[Uniformly proper solutions]\label{defn proper}
Let $(M,g)$ be a complete Riemannian surface and $T \in (0,\infty)$. We say that $\gamma : \mathbb{R} \times [0,T] \rightarrow M$ is a uniformly proper solution to CSF (in $M$) if
\begin{enumerate}[label=\roman*)]
\item $\gamma : \mathbb{R} \times [0,T] \rightarrow M$ is a continuous proper map.
\item $\gamma(\cdot, t) : \mathbb{R} \rightarrow M$ is a smooth proper embedding $\forall t \in (0,T]$.
\item $\gamma$ is smooth and solves (\ref{eqn:tang}) on $\mathbb{R} \times (0,T)$.
\end{enumerate}
\end{defn}

\begin{rem}\label{remark arm}
We restrict ourselves to solutions that are proper as a map on space-time to avoid tangential reparameterisations which get arbitrarily bad as $t$ goes to zero. The details of the following construction are given in the author's PhD thesis \cite{PhDthesis}.\par

Consider a cusp at infinity formed by the $x$-axis and another disjoint curve $L$ asymptotic to the $x$-axis, such that for $x_0$ sufficiently large, the region between $L$ and the $x$-axis intersected with the half-space $\{ x \geq x_0 \}$ has finite area. Then we can find a smooth proper solution to CSF starting from this cusp. That is, there exists a smooth solution $\gamma : \mathbb{R} \times (0,\infty) \rightarrow \mathbb{R}^2$ to CSF such that
\begin{itemize}
    \item The image of $\gamma(\cdot,t)$ converges to the $x$-axis locally uniformly over $(-\infty,0)$ as $t \searrow 0$.
    \item The image of $\gamma(\cdot,t)$ converges to $L$ locally uniformly over $(0,\infty)$ as $t \searrow 0$.
    \item $\gamma(0,t) \rightarrow \infty$ as $t \searrow 0$.
\end{itemize}
 By choosing a suitable time-dependent reparameterisation of $\gamma$, we can lose $L$ as initial data. That is, there exists a reparameterisation $\tilde{\gamma} : \mathbb{R} \times (0,\infty) \rightarrow \mathbb{R}^2$ of $\gamma$, such that the image of $\tilde{\gamma}(\cdot,t)$ converges to the $x$-axis locally uniformly over $\mathbb{R}$ as $t \searrow 0$. In particular, we have a continuous map $\tilde{\gamma} : \mathbb{R} \times [0,\infty) \rightarrow \mathbb{R}^2$ such that
\begin{itemize}
\item $\tilde{\gamma}(\cdot, t) : \mathbb{R} \rightarrow \mathbb{R}^2$ is a smooth proper embedding $\forall t \in (0,\infty)$.
\item $\tilde{\gamma}$ is smooth and solves (\ref{eqn:tang}) on $\mathbb{R} \times (0,\infty)$.
\item $\tilde{\gamma}(\cdot,0)$ is a parameterisation of the $x$-axis, but for any $t>0$, $\Ima(\tilde{\gamma}(\cdot,t)) = \Ima(\gamma(\cdot,t))$ isn't the $x$-axis.
\end{itemize}
This example shows that the family of curves being uniformly proper (in time) is a necessary condition to impose on a class of solutions in which you expect uniqueness. Moreover, requiring our solution to be uniformly proper is also sufficient for the usual avoidance principle with closed curves \ref{avoid}.
\end{rem}

Tychonoff found non-zero solutions of the heat equation on Euclidean space with zero initial data \parencite{tychonoff1935theoremes}. Despite curve shortening flow being a non-linear geometric analogue to the heat equation, the non-linearity of the equation affects the diffusion term, dampening large perturbations at spatial infinity as they propagate inwards. A simple illustration of this phenomenon is the following.

\begin{eg}\label{ex1}
Consider the $x$-axis in the plane. Since this is a geodesic, there exists the static solution which remains stationary under CSF. Consider any uniformly proper solution to CSF starting from the $x$-axis. Given a closed circle embedded away from the $x$-axis, by the avoidance principle \ref{avoid}, our solution must avoid this shrinking circle as they simultaneously evolve under CSF. Choosing our initial circle to be sufficiently large, we can make it shrink arbitrarily slowly. In particular, taking sufficiently large circles lying above and below the $x$-axis, we can trap our solution at any later time in an arbitrarily small tubular neighbourhood of the $x$-axis. Thus, our solution must coincide with the static solution. 
\end{eg}

\begin{defn}\label{defn unique}
Let $(M^2,g)$ be a complete Riemannian surface. We say that CSF is unique on $(M,g)$ if, for any pair of uniformly proper solutions $\gamma_i : \mathbb{R} \times [0,T_i] \rightarrow M$ to CSF (see definition \ref{defn proper}) with the same initial data
\begin{equation*}
\gamma_1(x,0) = \gamma_2(x,0), \quad \forall x \in \mathbb{R},
\end{equation*}
then their images agree wherever they are both defined
\begin{equation*}
\Ima(\gamma_1(\cdot,t)) = \Ima(\gamma_2(\cdot,t)), \quad \forall t \in [0,T],
\end{equation*}
with $T := \min \{ T_1,T_2 \}$.
\end{defn}
The following is a well known conjecture.
\begin{conj}
CSF is unique on the flat plane (see definition \ref{defn unique}).
\end{conj}

Daskalopoulos and Saez proved a uniqueness result for entire graphical solutions to $(\ref{eqn:tang})$ in the plane \cite{daskalopoulos2021uniqueness}. We note that every graphical solution is a uniformly proper solution, and so the conjecture could be seen as one way to generalise this result. Despite the conjecture potentially ruling out non-uniqueness in the flat plane, we will show that for other ambient surfaces CSF can be non-unique.\par

Ilmanen remarks that for surfaces with no lower curvature bound, it is possible for curves to bloom at infinity and rush inwards under CSF \cite[Remark 3.6]{ilmanen1994elliptic}. Drawing parallels again with the heat equation, this property is analogous to stochastic incompleteness, whereby heat is allowed to instantly escape at infinity.
Returning our attention to the example \ref{ex1}, we see that for surfaces that allow curves to bloom at infinity, our geometric proof of uniqueness
now fails.

\begin{eg}
Let $g = dx^2 + e^{2\phi(x)} dy^2$ be a complete metric on the plane $\mathbb{R}^2$. Consider any uniformly proper solution starting from the $x$-axis. Since the $x$-axis is still a geodesic with respect to this metric, we would like to show our solution is the static solution by the same reasoning as in example \ref{ex1}. However, we suppose that $g$ is chosen in such a way that vertical lines can bloom at infinity. That is, for some $T>0$ there exists a smooth function $x : (0,T) \rightarrow \mathbb{R}$ with $\lim_{t \searrow 0} x(t) = \infty$ and $\gamma : \mathbb{R} \times (0,T) \rightarrow \mathbb{R}^2$, with $\gamma(s,t) := (x(t) , s)$ in Cartesian coordinates, a proper solution to CSF on $(\mathbb{R}^2,g)$. Then, by the avoidance principle \ref{avoid}, every closed curve moving under CSF is instantly pulled in from infinity. More precisely, for any closed curve $\eta$ solving CSF on $(\mathbb{R}^2,g)$, we have that
\begin{equation*}
    \Ima(\eta(\cdot,t)) \subseteq \{ x \leq x(t) \}, \quad \forall t >0,
\end{equation*}
and we no longer have control of the non-closed solution at infinity.
\end{eg}
The main aim of this paper is to prove the following theorem, which asserts that for a choice of metric which allows curves to bloom at infinity, we do in fact have non-static solutions to CSF starting from a proper geodesic.
\begin{thm}\label{thm1}
There exists a smooth, complete metric $g = dx^2 + e^{2\phi(x)} dy^2$ on the plane and a uniformly proper solution $\gamma : \mathbb{R} \times [0,1] \rightarrow \mathbb{R}^2$ of CSF in $(\mathbb{R}^2,g)$ with initial condition $\gamma(\cdot,0)$ a parameterisation of the $x$-axis, such that for all $t>0$, the curve $\Ima(\gamma(\cdot,t))$ is not the $x$-axis.
\end{thm}
In particular, this gives the following corollary.
\begin{cor} There exists a Riemannian surface $(M^2,g)$ on which CSF is not unique (see definition \ref{defn unique}).
\end{cor}

Within the class of rotationally symmetric metrics on the plane, we are able to formulate a precise definition for blooming at infinity.\par

Given the usual action of the orthogonal group $O(2)$ on $\mathbb{R}^2$, consider a complete smooth $O(2)$-invariant metric $g$ on the plane. In polar coordinates $(r,\theta)$, the metric has the form 
\begin{equation}\label{eqn polar metric}
    g = dr^2 + e^{2\phi(r)} d\theta^2
\end{equation}
for some smooth warping function $\phi : (0,\infty) \rightarrow \mathbb{R}$. Under equation (\ref{eqn:tang}), the radii of the geodesic circles $\partial B_R:= \{ (R,\theta) : \theta \in S^1 \}$ solve the ODE
\begin{equation}\label{eqn symm circle}
    \frac{\partial R}{\partial t}(t) = -\frac{\partial\phi}{\partial r}(R(t)).
\end{equation}
We characterise such a metric to allow blooming at infinity if solutions to this ODE can come in from infinity in finite time:
\begin{defn}[Blooming at infinity]\label{defn bloom}
Consider the plane $(\mathbb{R}^2,g)$ equipped with a complete smooth $O(2)$-invariant metric, so that in polar coordinates it has the form $g = dr^2 + e^{2\phi(r)} d\theta^2$ as in (\ref{eqn polar metric}). We say that $g$ allows blooming at infinity if there exists $T \in (0,\infty)$ and a solution $R: (0,T) \rightarrow (0,\infty)$ to the ODE (\ref{eqn symm circle}) such that $R(t) \rightarrow \infty$ as $t \searrow 0$. If no such solution exists, we say that $g$ does not allow blooming at infinity.
\end{defn}

Within the class of smooth complete $O(2)$-invariant metrics which have non-positive curvature, we prove that if a metric does not allow blooming at infinity, then with respect to this metric we have uniqueness for uniformly proper solutions to CSF which start from a radial geodesic.

\begin{restatable}[Uniqueness of radial geodesics]{thm}{uniq}\label{thm2}
Consider a complete smooth $O(2)$-invariant metric with non-positive curvature on the plane. Let $\gamma : \mathbb{R} \times [0,T] \rightarrow \mathbb{R}^2$ be a uniformly proper solution to CSF starting from the $x$-axis. If $g$ does not allow blooming at infinity then $\gamma$ is the static solution to CSF.
\end{restatable}

In light of the previous theorems, we tentatively make the following uniqueness conjecture, claiming that within our special class of metrics, the only obstruction to uniqueness under CSF starting from any initial data is precisely blooming at infinity.

\begin{conj}
Let $(\mathbb{R}^2,g)$ be the plane equipped with a complete smooth $O(2)$-invariant metric with non-positive curvature. Then CSF is unique on $(\mathbb{R}^2,g)$ (see definition \ref{defn unique}) iff $g$ does not allow blooming at infinity (see definition \ref{defn bloom}).
\end{conj}

\subsection{Outline of the paper}
Sections \ref{chapter2}, \ref{chapter3} \& \ref{chapter4} are dedicated to proving theorem \ref{thm1}, with section \ref{chapter5} containing the proof of theorem \ref{thm2}.\par

We first aim to construct a non-compact solution to CSF via a limit of compact solutions. In section \ref{chapter2} we consider a Dirichlet problem over a compact interval using the graphical formulation of CSF on the interior. We choose suitable auxiliary conditions so that our solution is initially constant for most of the interval, but the value it takes changes near the boundary. The solution to such a Dirichlet problem will be smooth and immortal. Taking larger and larger compact intervals, the corresponding solutions to these Dirichlet problems will give a sequence of compact solutions which initially agree with the $x$-axis on larger and larger regions of the real line, but which deviate from the $x$-axis by a prescribed amount further and further out. In order to pass to a limit we require local uniform regularity for our sequence of solutions.\par

In section \ref{chapter3} we use a foliation argument to prove local uniform $C^1$-bounds, which then extends to local uniform $C^k$-bounds via parabolic regularity. This then allows us to pass to a limit, which will be a smooth immortal solution to CSF starting from the $x$-axis.\par

So far, the properties of our metric have not been crucial in the construction. In fact, for any choice of warping function, the ideas given so far could be used to produce a solution to CSF starting from $x$-axis. In section \ref{chapter4} we utilise our specific choice of metric to construct a barrier which moves in from infinity in finite time, and which pushes our solution away from the $x$-axis instantaneously.\par

Finally, the proof of theorem \ref{thm2} in section \ref{chapter5} is essentially a modified version of the barrier argument seen in example \ref{ex1}. We show that at an arbitrarily large time and for arbitrarily thin convex neighbourhoods of our geodesic, we can find closed solutions to CSF which not only exist until this time, but also lies arbitrarily far out inside our convex neighbourhood at this time. By applying the avoidance principle to our uniformly proper solution and the closed curves mentioned above, we force the image of our uniformly proper solution to agree with the $x$-axis at any positive time.

\subsection{Notation}
Given $\Omega \Subset \mathbb{R}$ and $T \in (0,\infty]$, we use the notation $\Omega_{T} := \Omega \times (0,T)$ to denote the parabolic rectangle and $\Gamma_{T}:= \left(\Omega \times \{0\}\right) \cup \left(\partial \Omega \times [0,T) \right)$ to denote its parabolic boundary. For each $j \in \mathbb{N}$ and $\alpha \in (0,1]$, we use $P^{j}(\Omega_T) := C^{2j,j} (\Omega_T)$ to denote the parabolic $C^j$ space and $P^{j,\alpha}(\Omega_T) := C^{2j,\alpha,j,\frac{\alpha}{2}} (\Omega_T)$ the corresponding parabolic Hölder space.

\section{Graphical curve shortening flow}\label{chapter2}

We begin by fixing a metric $g:= dx^2 + e^{2\phi(x)} dy^2$ on the plane for some smooth $\phi :\mathbb{R} \rightarrow  \mathbb{R}$, where $(x,y)$ are the standard cartesian coordinates. We will consider graphical solutions to CSF with respect to this metric. That is, we suppose we have a curve satisfying CSF such that you can either write $x(y,t)$ as a function of $y$ and $t$, or $y(x,t)$ as a function of $x$ and $t$. Since $\phi$ is independent of $y$, a solution to CSF remains a solution after a translation along the $y$-axis. In the case $x(y,t)$, translating along the $y$-axis corresponds to a horizontal translations of the graph, and in the case $y(x,t)$, a vertical translation. As such, we refer to these cases as horizontal or vertical graphs respectively. It is a routine calculation to show that the geodesic curvature $\kappa$ of our curve is given by
\begin{align}
&\kappa = \frac{\phi' e^{\phi} (e^{2\phi} + 2x_y^2) - e^{\phi} x_{yy}}{(e^{2\phi} + x_y^2)^{\frac{3}{2}}}, \quad \text{for a horizontal graph}\  x(y,t),\label{k,s}\\
&\kappa = \frac{\phi' e^{\phi} y_x (y_x^2 e^{2\phi} + 2) + e^{\phi} y_{xx}}{(1 + e^{2\phi} y_x^2)^{\frac{3}{2}}}, \quad
\text{for a vertical graph}\ y(x,t).\label{k,r}
\end{align}
Substituting these equations into (\ref{eqn:tang}) gives the graphical formulations for CSF on $(\mathbb{R}^2,g)$
\begin{equation}\label{CSF,1}
    x_t = \frac{x_{yy}}{e^{2\phi} + x_y^2} - \phi'(x) \left( 1 + \frac{x_y^2}{e^{2\phi} + x_y^2} \right) = \frac{\partial}{\partial y}\left( e^{-\phi} \tan^{-1}(x_{y} e^{-\phi}) \right) + \phi'(x) \left( x_y e^{-\phi} \tan^{-1}(x_y  )e^{-\phi} -1 \right),
\end{equation}
\begin{equation}\label{CSF,2}
    y_t = \frac{y_{xx}}{1 + e^{2\phi} y_x^2} + y_x \phi'(x)  \left(1 + \frac{1}{1 + e^{2\phi} y_x^2} \right) = \frac{\partial}{\partial x} \left( e^{-\phi} \tan^{-1}(y_x e^\phi)  \right) + \phi'(x) \left( y_x + e^{-\phi} \tan^{-1}(y_x e^\phi) \right).
\end{equation}
Since we refer to these PDEs throughout the rest of the paper, we introduce the notation
\begin{equation*}
\mu(p):= \frac{1}{1 + p^2 e^{2\phi}} \in (0,1], \quad  \nu(p) := \frac{1}{e^{2\phi} + p^2} \in (0,e^{-2\phi}].
\end{equation*}
With this notation, we have the quasi-linear operators
\begin{align*}
\mathcal{H}(x) &:= x_t - \nu(x_y) x_{yy} + \phi'(x) (1+\nu(x_y) x_y^2),\\
\mathcal{V}(y) &:= y_t - \mu(y_x) y_{xx} - \phi'(x) (1+\mu(y_x)) y_x,
\end{align*} 
so that equations (\ref{CSF,1}) and (\ref{CSF,2}) become $\mathcal{H}=0$ and $\mathcal{V}= 0$ respectively. The following theorem is the graphical version of theorem \ref{thm1}. 
\begin{thm}\label{main thm}
There exists a smooth, even function $\phi:\mathbb{R} \rightarrow \mathbb{R}$ and a continuous function $y: \mathbb{R} \times [0,\infty) \rightarrow [-1,1]$ such that
\begin{enumerate}
\item [(i)]  $y(\cdot , 0) \equiv 0$ on $\mathbb{R}$.
\item[(ii)] $y(\cdot,t)$ is an increasing odd function $\forall t \in (0,\infty)$.
\item [(iii)]  $y$ is smooth and satisfies $\mathcal{V}(y)=0$ on $\mathbb{R} \times (0,\infty)$.
\item [(iv)] $y(\cdot,t)$ instantly peels away at infinity
\begin{equation*}
\forall \epsilon, t > 0, \ \exists x_0>0 \quad \textnormal{such that} \quad y(x,t) > 1 - \epsilon, \quad \forall x > x_0.
\end{equation*}
\end{enumerate}
\end{thm}
To prove theorem \ref{thm1} it suffices to prove theorem \ref{main thm}, as theorem \ref{thm1} follows from theorem \ref{main thm} after setting $\gamma(x,t) := (x,y(x,t))$.

\subsection{Choosing our metric}\label{chapter choosing metric}
Before continuing we choose $\phi$ to be used in theorem \ref{main thm}.
\begin{restatable*}{lem}{warp}\label{exist of phi}
There exists a smooth function $\phi :\mathbb{R} \rightarrow  \mathbb{R}$ such that
\begin{enumerate}
\item $\phi$ is an even function, which is increasing on $(0,\infty)$.
\item $\phi(x) = 0$ for all $x \in [0,1]$.
\item $\phi'(x)>0$ and $\phi'$ is strictly increasing on $(1,\infty)$.
\item $\phi'(x) < \frac{1}{2}$ for all $x \in (1,\frac{3}{2})$.
\item $\phi'(x) = x^2$ for all $x \geq 2$.
\end{enumerate}
\end{restatable*}
The construction of $\phi$ is just an exercise in choosing suitable bump functions. See \ref{exist of phi} for details.\par
Other than the growth rate at infinity, our choices for $\phi$ are not crucial, but instead help reduce the technicality of our arguments. The last condition however is essential. The rapid growth of $\phi$ for large $x$ is what allows curves to bloom at infinity. Given $c > 0$, consider the unique maximal solution to $\mathcal{H}=0$ with initial condition $c$. We note that this solution remains constant in $y$ at all later times, and so the corresponding curves will always be straight lines parallel to the $y$-axis. We denote this solution by $c(t)$. Since the equation $\mathcal{H}(c)=0$ is equivalent to the ODE
\begin{equation*}
    \frac{\partial c}{\partial t}= -\phi'(c(t)) \leq 0,
\end{equation*}
the corresponding curves are translating towards the $y$-axis under CSF. By our choice of $\phi$, for large $c$ and small $t$ we have the explicit formula $c(t) = (c^{-1} + t)^{-1}$. Taking $c \rightarrow \infty$, for small positive $t$, we have the explicit solution $\zeta(t) = t^{-1}$ to the equation $\mathcal{H}(\zeta)=0$. Geometrically this means that lines parallel to the $y$-axis fly in from infinity in finite time under CSF.

\subsection{Graphical geodesics}\label{chapter geodesics}
Given our choice of $\phi$, we can consider what the geodesics in our space now look like. Since geodesics are invariant under the flow, they are useful barriers. Setting $\kappa = 0$ in equation (\ref{k,r}) yields the  first order ODE
\begin{equation*}
(y_x e^\phi)_x + \phi'(x) \cdot (y_x e^\phi) \cdot (1 + y_x^2 e^{2\phi}) = 0.
\end{equation*}
We can solve this equation for all $x \in \mathbb{R}$ to give solutions, for each constant $m \in (-1,1)$
\begin{equation*}
y_x  = \frac{m}{e^\phi\sqrt{e^{2\phi} - m^2}}.
\end{equation*}
Thus, we can parameterise the proper geodesics with a vertical graphical representation by
\begin{equation*}
\{ \sigma_{m,h} : \mathbb{R} \rightarrow \mathbb{R} \ \vert \ m \in (-1,1),  h \in \mathbb{R} \},
\end{equation*}
where
\begin{equation*}
\sigma_{m,h}(x) := h + \int_0^x \frac{m}{e^{\phi(s)} \sqrt{e^{2\phi(s)} - m^2}} ds, \quad \forall x \in \mathbb{R}.
\end{equation*}
We note that $\sigma_{0,h}$ parameterises the horizontal line $\{y=h\}$. For $m \neq 0$ however, $\sigma_{m,h}' \neq 0$ everywhere and the corresponding geodesic also has a horizontal graphical representation $\eta_{m,h} := (\sigma_{m,h})^{-1}$.

\subsection{CSF Dirichlet problems}\label{chapter dir probs}
As a first step to proving theorem \ref{main thm}, we consider the following CSF Dirichlet problems, the first of which is for vertical graphs. The solutions to this Dirichlet problem will be suitable approximations to our desired solution in theorem \ref{main thm}.
\subsubsection{Vertical graphs}
We start by defining our initial data. Fix $\chi: [0,1] \rightarrow [0,1]$ a smooth, decreasing cut off function such that $\chi \equiv 1$ on $[0,\frac{1}{4})$, $\chi \equiv 0$ on $(\frac{3}{4},1]$ and $ \chi' > -4$ on $[0,1]$. For each $n \in \mathbb{N}$, define the function $Y_{n} :[-n,n] \rightarrow [-1,1]$ to be the unique odd function such that,
\begin{equation*}
Y_{n}(x) := \begin{cases}
0 &: x \in [0, n-1]\\
\chi(x+1-n)) &: x \in [n-1,n]
\end{cases}
\end{equation*}
We trivially extend this to a function on $[-n,n] \times [0,\infty)$ by making it constant in time so that using it for our auxiliary conditions will correspond to fixing the endpoints of our curve. For each $n \in \mathbb{N}$ and $s \in (0,\infty]$ consider the Dirichlet problem for vertical graphs
\begin{equation}\label{dir,1}
V_n(s) := \begin{cases}
\mathcal{V}(y) =0 \quad &\text{in} \ \Omega_s \\
y = Y_{n} \quad &\text{on} \ \Gamma_s \\
\end{cases}
\end{equation}
where $\Omega:=(-n,n)$. Applying theorem \ref{max sol} and theorem \ref{reg of sol} to equation (\ref{CSF,2}), $\forall n \in \mathbb{N}$
$\exists \ T_n \in (0,\infty]$ and a unique maximal solution $y_n : [-n,n] \times [0,T_n) \rightarrow \mathbb{R}$ satisfying
\begin{enumerate}\label{eqn max soln y}
\item[(A)] $y_n \in C^\infty(\Omega_s)$,  $\forall s \in (0,T_n)$.
\item[(B)] $y_n$ solves $V_n(s)$, $\forall s \in (0,T_n)$.
\item[(C)] If $T_n < \infty$ then $\limsup_{s \rightarrow T_n} \abs{y_n(\cdot,s)}_{C^1([-n,n])} = \infty$.
\end{enumerate}
Finally by the symmetries of $Y_{n},\phi$ and $\mathcal{V}$, $y_n(\cdot,t)$ is an odd function for all $t \in [0,T_n)$.
\subsubsection{Horizontal graphs}
Although we have our sequence of vertical graphs $y_n$ to approximate an entire solution, we also need to switch gage and consider a Dirichlet problem for horizontal graphs. These horizontal graphical solutions will foliate regions of the plane and will be used in section \ref{chapter3} to show local gradient bounds for the sequence $y_n$.\par

Recall from the discussion in section \ref{chapter choosing metric} that we have the unique maximal solutions $c(t)$ to $\mathcal{H}=0$ starting from the constant initial condition $c>0$. Unlike for vertical graphs where we keep the endpoints fixed, we will instead use these solutions $c(t)$ for the auxiliary data. For each $c>0$ and $s \in (0,\infty]$ consider the Dirichlet problem
\begin{equation}
H_{c}(s) := \begin{cases}
\mathcal{H} (x) = 0 \quad &\text{in} \ (0,1) \times (0,s) \\
x = c \quad & \text{on} \ [0,1] \times \{0\}\\
x(0,t)=c(t),  \ x(1,t) = c \quad &\forall t \in (0,s).
\end{cases}
\end{equation}
so that on the parabolic wall $\{ y=1 \}$ the endpoint of the curve is fixed, but on the parabolic wall $\{ y=0 \}$ the endpoint is moving down at the same rate as the constant solution.
Applying theorem \ref{max sol} and theorem \ref{reg of sol} to equation (\ref{CSF,2}), $\forall c>0$, $\exists \ T_c \in (0,\infty]$ and a unique maximal solution $g_c : [0,1] \times [0,T_c) \rightarrow [0,\infty)$ satisfying,
\begin{enumerate}\label{eqn max soln c}
\item[(A)] $g_c \in C^\infty((0,1)\times(0,s))$,  $\forall s \in (0,T_c)$.
\item[(B)] $g_c$ solves $H_c(s)$, $\forall s \in (0,T_c)$.
\item[(C)] If $T_c < \infty$ then $\limsup_{s \rightarrow T_c} \abs{g_c(\cdot,s)}_{C^1([0,1])} = \infty$.
\end{enumerate}

Before taking a limit of the sequence $y_n$, we need to be sure that they exist for a uniform amount of time. We shall in fact show that each of the solutions is immortal. This is not obvious a priori; solutions to the Dirichlet problem converge towards a geodesic between the fixed endpoints which need not be graphical. In order to show that the solutions are immortal, it suffices to show that the solutions and their gradients cannot blow up in finite time, as otherwise this would contradict (C) in section \ref{eqn max soln y}.

\subsection{Uniform bounds}
Using barriers we can bound the region on which our solutions exist. This will give $C^0$-bounds everywhere as well as $C^1$-bounds on the parabolic boundary.
\subsubsection{Vertical graphs}
\begin{lem}\label{C0 radial}
For each $n \in \mathbb{N}$, let $y_n:[-n,n]\times[0,T_n)\rightarrow \mathbb{R}$ be the maximal solution to the Dirichlet problem constructed in Chapter \ref{chapter dir probs}. Then for each $t \in [0,T_n)$, the graph of $x \mapsto y_n(x,t)$ is contained in the parallelogram
\begin{equation}\label{radial region}
\{ (x,y) \in [-n,n] \times [-1,1] :  1 + 4(x-n) \leq y \leq -1 + 4(x+n) \} \}.
\end{equation}
\end{lem}
\begin{proof}
Using that $\phi'(x) \geq 0$ for $x \geq 0$ and $\mu(p) > 0$ for any $p \in \mathbb{R}$, we have that
\begin{equation*}
\mathcal{V}(4x) = -4\phi'(x)(1+\mu(4)) \leq 0.
\end{equation*}
Using $0 \leq \frac{\partial y_n}{\partial x} (\cdot,0) \leq 4$ and that $y_n(\cdot,t)$ is an odd function for each $t$
\begin{align*}
1 + 4(x-n) \leq y_n(x,0), &\quad \forall x \in [0,n],\\
1 + 4(0-n) \leq 0 = y_n(0,\cdot), &\quad 1 + 4(n-n) = 1 = y_n(n,\cdot).
\end{align*}
Hence $1 + 4(x-n)$ is a lower barrier to $y_n$ over $[0,n]$. By symmetry, $-1 + 4(x+n)$ is an upper barrier over $[-n,0]$.  Combining with the upper and lower barriers $1$ and $-1$ gives the result.
\end{proof}
\subsubsection{Horizontal graphs}
We will also need that the horizontal graphs exist for a uniform amount of time. As with the vertical graphs, we shall show that each maximal solution $g_c$ is immortal.
\begin{lem}\label{C0 angular}
For each $c>0$, let $g_c:[0,1]\times[0,T_c) \rightarrow [0,\infty)$ be the maximal solution to the Dirichlet problem constructed in section \ref{chapter dir probs}. Then there exists a constant $m \in (0,1)$ depending on $c$ such that,
for each $t \in [0,T_c)$, the graph of $y \mapsto g_c(y,t)$ is contained in the region
\begin{equation}\label{spiral region}
 \{ (x,y) \in [c(t),c] \times [0,1] : \eta_{m,0}(y) \leq x \leq c(t)(1-y) + cy \},
\end{equation}
where $\eta_{m,0}$ refers to the horizontal geodesic constructed in section \ref{chapter geodesics}.
\end{lem}
\begin{proof}
Using that $\phi'(x)\geq 0$ is increasing for $x>0$, $c(t)>0$ is decreasing, and $\nu(p) > 0$ for any $p \in \mathbb{R}$, we have that
\begin{align*}
\mathcal{H}(c(t)(1-y) + cy ) &\geq c'(t)(1-y) + \phi'(c(t)(1-y) + c y) \\
&\geq \phi'(c(t))(y-1) + \phi'(c(t)) \geq 0.
\end{align*}
So $c(t)(1-y) + cy $ is a supersolution. Moreover
\begin{align*}
c(0)(1-y) + cy = c, \quad
c(t)(1-0) + c\cdot0 = c(t), \quad
c(t)(1-1) + c\cdot1 = c.
\end{align*}
Hence $c(t)(1-y) + cy $ is an upper barrier to $g_c$. For a lower barrier to $g_c$, choose $m \in (0,1)$ such that $\sigma_{m,0}(c) = 1$. Since $m \neq 0$, $\sigma_{m,0}$ is invertible, and we have the horizontal graph $\eta_{m,0}$ which is an increasing geodesic with $\eta_{m,0}(0) = 0$ and $\eta_{m,0}(1) = c$.
\end{proof}
\subsection{Preservation of monotonicity}

Now that we have good control of the solutions on the parabolic boundary, we will use the maximum principle to give $C^1$-control on the interior. Before we can do this however, we need one final ingredient. Notice that the initial data is monotonic for all of the Dirichlet problems. We now show that this monotonicity is always preserved.

\subsubsection{Vertical graphs}\label{chapter mono vert}
\begin{prop}\label{mono radial}
For each $n \in \mathbb{N}$, let $y_n:[-n,n]\times[0,T_n)\rightarrow \mathbb{R}$ be the maximal solution to the Dirichlet problem constructed in section \ref{chapter dir probs}. Then $y_n(\cdot,t)$ is strictly increasing for each $t \in (0,T_n)$. Moreover, the gradient is strictly positive away from the initial time
\begin{equation*}
\frac{\partial y_n}{\partial x}(x,t) > 0, \quad \forall (x,t) \in [-n,n] \times (0,T_n).
\end{equation*}
\end{prop}
\begin{proof}
For any $h \in [-1,1]$ we apply the the intersection principle \ref{int princ} to the function $y_n$ and the constant solution $h$ to get that the intersection number between them is decreasing in time and finite for positive time. By the intermediate value theorem, there is always at least one intersection. Fix $(x,t) \in [-n,n] \times (0,T_n)$ and $h \in [-1,1]$ such that $y_n(x,t) = h$. Consider those values of $h$ for which the initial intersection number is one
\begin{equation*}
H:= \{ h \in [-1,1] : \abs{y_n(\cdot,0)^{-1}(h)} = 1 \}.
\end{equation*}
If $h \in H$, then \ref{int princ} implies that there is always exactly one intersection between $h$ and $y_n(\cdot,t)$ for all $t \in [0,T_n)$. Moreover, the intersection point is always transverse, so if $h \in H$ then $\frac{\partial y_n}{\partial x}(x,t) > 0$ by \ref{fact}. For $h \notin H$, we can find $x_1 \neq x_2 \in Y_{n}^{-1}(h)$. Since $Y_{n}$ is increasing, $[x_1,x_2] \subseteq Y_{n}^{-1}(h)$ and hence $Y_{n}^{-1}(h)$ has positive measure. This implies that $H$ is dense in $[-1,1]$.\par

We will now use the density of $H$ in $[-1,1]$ to show that $y_n(\cdot,t)$ is increasing for any $t \in [0,T_n)$. Fix $x_0 \in [-n,n]$ and let $h \in H$ with $h < y_n(x_0,t)$. Suppose $y_n(\cdot,t) < h$ at some point past $x_0$, so that there is a well defined element
\begin{equation*}
x_1:= \inf \{ z \in [x_0,n] : y_n(z,t) \leq h \}.
\end{equation*}
By the continuity of $y_n(\cdot,t)$, we have that $y_n(x_1,t) = h$. Also $h< y_n(x_0,t) \leq 1 = y_n(n,t)$, which means that $x_1 \in (x_0,n)$. However this is a clear contradiction to the fact that $\frac{\partial y_n}{\partial x}(x_1, t) > 0$, meaning that we must have $y_n(x,t) > h$ for every $x>x_0$, which by the density of $H$ in $[-1,1]$, allows us to conclude that $y_n(\cdot,t)$ is increasing.\par

Finally, to show that the gradient is positive everywhere, we repeat the argument from before but at positive times $s \in (0,T_n)$ instead of at time $0$, and then use that the intersection number is finite. If $\abs{ y_n(\cdot,s)^{-1}(h)} > 1$, then since $y_n(\cdot,s)$ is increasing $y_n(\cdot,s)^{-1}(h)$ has positive measure, contradicting the fact that there are only finitely many intersections \ref{int princ}. We conclude that $\abs{ y_n(\cdot,s)^{-1}(h)} =1$ and this single intersection point between $y_n(\cdot,s)$ and $h$ is transverse. As before, a transverse intersection point implies $\frac{\partial y_n}{\partial x}(x,s) > 0$ by \ref{fact}.
\end{proof}
\subsubsection{Horizontal graphs}
\begin{prop}\label{mono angular}
For each $c>1$, let $g_c:[0,1]\times[0,T_c) \rightarrow [0,\infty)$ be the maximal solution to the Dirichlet problem constructed in section \ref{chapter dir probs}. Then $g_c(\cdot,t)$ is strictly increasing for each $t \in (0,T_c)$. Moreover, the gradient is strictly positive away from the initial time
\begin{equation*}
\frac{\partial g_c}{\partial y}(y,t) > 0, \quad \forall (y,t) \in [0,1] \times (0,T_c).
\end{equation*}
\end{prop}
\begin{proof}
We repeat the proof of proposition \ref{mono radial} but replace constant solutions with the solutions $c(t)$ from section \ref{chapter choosing metric}.\par

Fix $(y,t) \in [0,1] \times (0,T_c)$ and suppose $g_c(y,t) > c(t)$. Then there exists some $\tau \in [0,t)$ such that $c(\tau) = g_c(y,t)$. In particular, the solutions $c(s)$ and $g_c(y,s+(t-\tau))$, defined on $[0,1] \times [0,T_c - (t-\tau))$, intersect at $(y,\tau)$. In order to apply the intersection principle, we note that
\begin{itemize}
\item $g_c(1, s + (t-\tau)) = c > c(s)$, $\forall s > 0$.
\item $g_c(0, s + (t-\tau)) = c(s+(t-\tau)) < c(s)$,  $\forall s \geq 0$.
\item $\frac{\partial g_c}{\partial y} (1,t-\tau) \geq  c - c(t) >0$, from the uniform bounds on $g_c$ (lemma \ref{C0 angular}).
\end{itemize}
So the intersection principle \ref{int princ} holds and there is a unique transverse point of intersection between $c(s)$ and $g_c(y,s+(t-\tau))$ for all $s \in [0,T_c-(t-\tau))$. In particular $\frac{\partial g_c}{\partial y} (y,t) > 0$ by \ref{fact}.\par
Using again the uniform bounds on $g_c$ (lemma \ref{C0 angular}), $g_c(\cdot,t)$ takes values in $[c(t) , c]$, and hence we have shown that $g_c(\cdot,t)$ is decreasing. The only thing left to deal with is the set of points
\begin{equation*}
\{ y \in [0,1] :  g_c(y,t) = c(t) \},
\end{equation*}
which is a closed interval containing $0$. Suppose now that $g_c(y,t) = c(t)$ and consider the intersection number for the solutions $g_c$ and $c$. We note that
\begin{itemize}
\item $g_c(0, s) = c(s)$, $\forall s > 0$.
\item $g_c(1, s) = c > c(s)$,  $\forall s > 0$.
\end{itemize}
Therefore the intersection principle \ref{int princ} holds and the number of intersections between between $g_c$ and $c$ is decreasing and finite on $(0,T_c)$. Combining this with $g_c(\cdot,t)$ increasing, we have that $g_c(\cdot,t) = c(t)$ only at $0$ for any $t>0$. We conclude that the intersections are always transverse and so $\frac{\partial g_c}{\partial y} (0,t) > 0$ by \ref{fact}.
\end{proof}
\subsection{Longtime existence of solutions}
We now have everything we need to show that the maximal solutions $y_n$ and $g_c$ to both classes of Dirichlet problems considered in section \ref{chapter dir probs} are immortal.
\subsubsection{Vertical graphs}
\begin{thm}
For each $n \in \mathbb{N}$, let $y_n:[-n,n]\times[0,T_n)\rightarrow \mathbb{R}$ be the maximal solution to the Dirichlet problem constructed in Chapter \ref{chapter dir probs}. Then $T_n = \infty$ and $y_n \in C^{\infty}( [-n,n] \times [0,\infty))$ for each $n \in \mathbb{N}$.
\end{thm}
\begin{proof}
Suppose $T_n < \infty$. Since our solution is smooth, we can differentiate (\ref{CSF,2}) to get the evolution equation for the gradient $v:= \frac{\partial y_n}{\partial x}$ on $(-n,n) \times (0,T)$
\begin{equation}\label{grad evo}
v_t = \mu(v) v_{xx} + \phi'(x) (1+\mu(v)) v_x + \phi''(x)(1+\mu(v))v - 2\mu(v)^2 e^{2\phi} (v_x + \phi'(x)v)^2 v.
\end{equation}
By proposition \ref{mono radial}, $v \geq 0$ for all time, and so $v$ satisfies the differential inequality
\begin{equation}\label{grad ineq}
v_t - \mu(v) v_{xx} - \phi'(x) (1+\mu(v)) v_x - \phi''(x)(1+\mu(v))v \leq 0.
\end{equation}
Furthermore, since $\phi''(x) (1+\mu(v)) \leq M$ on the finite interval $[-n,n]$ for some constant $M$, by the maximum principle \ref{linear max princ}
\begin{equation}\label{grad lower bd}
v \leq e^{MT_n} \sup_{\Gamma_{T_n}} \left( v \right).
\end{equation}
By lemma \ref{C0 radial}, $v(\pm n,t) \leq 4$ and we have the following contradiction
\begin{equation*}
\limsup_{s \rightarrow T_n} \ \abs{ y_n (\cdot,s)}_{C^1([-n,n])} \leq 1 + 4e^{MT_n} < \infty. \qedhere
\end{equation*}
\end{proof}
\subsubsection{Horizontal graphs}
\begin{thm} For each $c>0$, let $g_c:[0,1]\times[0,T_c) \rightarrow [0,\infty)$ be the maximal solution to the Dirichlet problem constructed in section \ref{chapter dir probs}. Then $T_c = \infty$ and $g_c \in C^{\infty}( [0,1] \times [0,\infty))$ for each $c>0$.
\end{thm}
\begin{proof}
Since $\phi \equiv 0$ for $\abs{x}\leq 1$, if $c \in (0,1]$ then $g_c$ is just the line static line $\{x=c\}$, so we may assume $c>1$. Suppose $T_c < \infty$. Since our solution is smooth, we can differentiate (\ref{CSF,1}) to get the evolution equation for the gradient $w:= \frac{\partial g_c}{\partial y}$ on $(0,1) \times (0,T_c)$
\begin{equation}\label{grad evo angular}
w_t = \nu(w) w_{yy}  + \left( 2 \phi'(x)^2 e^{2\phi} \nu(w)  - \phi''(x) (1 + \nu(w) w^2) \right) w - 2\nu(w)^2 (w_{y} + \phi'(x) e^{2\phi})^2 w.
\end{equation}
By proposition \ref{mono angular}, $w \geq 0$ for all time, and so $w$ satisfies the differential inequality
\begin{equation}\label{grad ineq angular}
w_t - \nu(w) w_{\theta \theta}  - 2 \phi'(x)^2 e^{2\phi} \nu(w)  w  \leq 0.
\end{equation}
Furthermore, since $2 \phi'(x)^2 e^{2\phi} \nu(w)  \leq M$ for some constant $M>0$ (we used that $g_c \leq c$),  we can apply the maximum principle \ref{linear max princ}
\begin{equation}\label{grad lower bd angular}
w \geq e^{MT_c} \sup_{\Gamma_{T_c}}  w.
\end{equation}
By lemma \ref{C0 angular}, $w(0,t) \leq c-c(T)$ and $w(1,t) \leq \eta_{m,0}'(1)$, giving the following contradiction
\begin{equation*}
\limsup_{s \rightarrow T_c} \ \abs{ g_c (\cdot,s)}_{C^1([0,1])} \leq c + e^{MT_C} \max\{c-c(T_c), \eta_{m,0}'(1)\} < \infty. \qedhere
\end{equation*}
\end{proof}

\section{Constructing an entire solution}\label{chapter3}
In the previous section, we constructed a sequence of continuous functions
\begin{equation*}
y_n : [-n,n] \times [0,\infty) \rightarrow [-1,1], \quad \forall n \in \mathbb{N},
\end{equation*}
such that
\begin{enumerate}
\item [(i)] $y_n$ is smooth on $[-n,n] \times (0,\infty)$ and solves $V_n(T)$ for any  $T > 0$.
\item [(ii)] $y_n(\cdot,t)$ is a strictly increasing odd function with positive gradient for any $ t>0$.
\item [(iii)] $y_n(x, 0) = 0$ for $   x \in [1-n, n-1]$.
\item [(iv)] $y_n(x,t)$ is a decreasing sequence in $n$ for any $(x,t) \in [0,n] \times [0,\infty)$.
\end{enumerate}
To see (iv) we note that, for $m \leq n$
\begin{align*}
&y_{n}(x,0) \leq y_{m}(x,0), \quad \forall x \in [0,m].\\
&y_n(0,t)= 0 =y_m(0,t), \ \text{and} \ y_n(m,t) \leq 1 = y_m(m,t), \quad \forall t > 0.
\end{align*}
(iv) then follows from the avoidance principle \ref{avoidance princ}.\par
These properties allow us to do several things:
\begin{itemize}
\item By (ii),  $y_n(\cdot,t)$ is invertible for each $n \in \mathbb{N}$ and $t>0$. Thus, we can change gage and consider the curves as horizontal graphs
\begin{equation*}
x_n : [-1,1] \times (0,\infty) \rightarrow [0,n] , \quad x_n(\cdot,t):= y_n(\cdot,t)^{-1}, \ \forall n \in \mathbb{N}.
\end{equation*}
\item By (i), the horizontal graphs $x_n$ are smooth and satisfy $\mathcal{H}(x_n) = 0$ on $[-1,1] \times (0,\infty)$.
\item (iv) allows us to take a limit of this sequence to get a function 
$y: \mathbb{R} \times [0,\infty) \rightarrow [-1,1]$,
\begin{equation*}
y(x,t) := \lim_{n\rightarrow \infty} y_n(x,t), \quad \forall (x,t) \in \mathbb{R} \times [0,\infty). 
\end{equation*}
\item By (iii), $y(\cdot,0) \equiv 0$ on $\mathbb{R}$ and by (ii), $y(\cdot , t)$ is an increasing odd function for any $t >0$. 
\end{itemize}
 We want to show that $y$ is smooth and solves $\mathcal{V}(y)=0$ on $\mathbb{R} \times (0,\infty)$. This will follow if we can show that the convergence of the sequence $y_n$ to $y$ is locally smooth on $\mathbb{R} \times (0,\infty)$. We also want to show that $y$ is continuous on $\mathbb{R} \times [0,\infty)$. However,  once we have shown local smooth convergence on $\mathbb{R} \times (0,\infty)$, continuity on $\mathbb{R} \times [0,\infty)$ follows from the monotonicity of each term in the sequence and the monotonicity of the limit.
 \begin{lem}
 If $y_n$ converges to $y$ locally smoothly on $\mathbb{R} \times (0,\infty)$, then $y \in C(\mathbb{R} \times [0,\infty))$.
 \end{lem}
 \begin{proof}
 Since $y(\cdot,0)$ is an odd function, it suffices to show that $y$ is continuous at any point $(x_0,0)$ with $x_0>0$. Fix $\epsilon > 0$ and choose $n$ large enough so that $y_n(x_0,0)=0$. Since $y_n$ is continuous at $(x_0,0)$, there exists some $\delta \in (0,x_0)$ such that 
 \begin{equation*}
     |y_n(x,t)| \leq \epsilon, \quad \forall (x,t) \in (x_0-\delta,x_0+\delta) \times [0,\delta).
 \end{equation*}
 As $y_n(x,t)$ is an increasing function in $x$ for any fixed $t$
 \begin{equation*}
     0 = y_n(0,t) \leq y_n(x,t) < \epsilon \quad \forall (x,t) \in (x_0-\delta,x_0+\delta) \times [0,\delta).
 \end{equation*}
 Finally, as $y_n(x,t)$ is a decreasing sequence in $n$
 \begin{equation*}
     0 \leq y(x,t) < \epsilon \quad \forall (x,t) \in (x_0-\delta,x_0+\delta) \times [0,\delta). \qedhere
 \end{equation*}
 \end{proof}

\subsection{Local gradient bounds}
The goal of this next section is to show the following theorem.
\begin{thm}[Local gradient bounds]\label{loc grad bd}
Fix $k,T>0$. Then there exists $M_1(k,T) >0$ such that
\begin{equation*}
\abs{y_n(\cdot,t)}_{C^1([-k,k])} \leq M_1, \quad \forall t \in [0,T], \quad \forall n \in \mathbb{N}.
\end{equation*}
That is, on any compact region of space time, our sequence has a uniform spatial $C^1$-bound.
\end{thm}

Since each $y_n$ is smooth, it suffices to prove theorem \ref{loc grad bd} for $n$ sufficiently large. So from now on, we may assume $n > k+1$.

The strategy we employ to prove theorem \ref{loc grad bd} is as follows:
\begin{itemize}
\item Foliate a region of space-time with curves of controlled gradient.
\item Show that at any time, each solution $y_n$ intersects each foliating curve only once.
\item  Conclude a uniform gradient bound over the foliated region using \ref{fact}. 
\end{itemize}
To begin, we shall show a uniform gradient bound at all times for our solutions over some compact subset of space strictly containing $[-1,1]$.
\begin{lem}[Local gradient bounds on a small spatial neighbourhood]\label{ssn grad bd} There exists $a > 1$ and $M_1 >0$ such that
\begin{equation*}
0 \leq \frac{\partial y_n}{\partial x}(x,t) \leq M_1 \quad \forall x \in [0,a], \ \forall t \in [0,T], \ \forall n \in \mathbb{N}.
\end{equation*}
In particular
\begin{equation*}
\abs{y_n(\cdot,t)}_{C^1([-a,a])} \leq M_1, \quad \forall t \in [0,T], \quad \forall n \in \mathbb{N}.
\end{equation*}
\end{lem}
The last line of this lemma rephrases it as the special case of theorem \ref{loc grad bd} for $k \leq a$. To prove lemma \ref{ssn grad bd}, we shall foliate $[-2,2] \times \mathbb{R}$ with the geodesics mentioned in section \ref{chapter geodesics}.  As such, the proof will require the following property of the geodesics.
\begin{claim}
For each $m \in (0,1)$, let $\sigma_{m,0}$ denote the graphical geodesic constructed in section \ref{chapter geodesics}. For $m$ sufficiently close to $1$
\begin{equation*}
\sigma_{m,0}(2) > \sigma_{m,0}(1)+1>2.
\end{equation*}
\end{claim}
\begin{proof}[Proof of claim]
Since $\phi'(x)<\frac{1}{2}$ for $x \in (1,\frac{3}{2})$, we note that $\phi(1+s) \leq \frac{s}{2}$ for $s \in (0,\frac{1}{2})$. For each $m \in (0,1)$
\begin{equation*}
\sigma_{m,0}(2) = \sigma_{m,0}(1) + \int\limits_1^2 \frac{m}{e^{\phi(s)}\sqrt{e^{2\phi(s)} - m^2}} \ ds > \sigma_{m,0}(1) + \int\limits_0^{\frac{1}{2}} \frac{m}{e^{\frac{s}{2}}\sqrt{e^s - m^2}} \ ds
\end{equation*}
Since $\int_0^{\frac{1}{2}} \frac{1}{e^{\frac{s}{2}}\sqrt{e^s - 1}} \ ds > 1$, for $m$ sufficiently close to $1$, $\sigma_{m,0}(2) > \sigma_{m,0}(1) + 1$.  We get the last inequality for free as the gradient of $\sigma_{m,0}$ is decreasing.
\end{proof}

\begin{proof}[Proof of lemma \ref{ssn grad bd}]
By the previous claim there exists $m \in (0,1)$ and $\epsilon > 0$ such that $\sigma_{m,0}(2) = \sigma_{m,0}(1) + 1 + 2\epsilon$.  Fix $a>1$ such that $\sigma_{m,0}(a) = \sigma_{m,0}(1) + \epsilon$.  We consider the geodesics $\sigma_{m,h}:\mathbb{R} \rightarrow \mathbb{R}$ for $\abs{h} \leq \sigma_{m,0}(1) + \epsilon$. Since
\begin{align*}
\sigma_{m,h}(2)  &\geq 1 + \epsilon > 1 \geq y_n(2,t),\\
\sigma_{m,h}(-2)  &\leq -1 - \epsilon < -1 \leq y_n(-2,t),
\end{align*}
by intersection principle \ref{int princ}, at each time $t \in [0,T]$ the curves $\sigma_{m,h}(\cdot)$ and $y_n(\cdot,t)$ intersect at a single point over $[-2,2]$. Fix $(x_0,t_0) \in [0,a] \times [0,T]$. Then for $h=-\sigma_{m,0}(1) - \epsilon$ we have $\sigma_{m,h}(x_0) \leq 0$, and for $h=\sigma_{m,0}(1) + \epsilon$ we have $\sigma_{m,h}(x_0) \geq 1$. Therefore, one of the $\sigma_{m,h}$ with $\abs{h} \leq \sigma_{m,0}(1) + \epsilon$ intersects $y_n$ at $(x_0,t_0)$, and we have
\begin{equation*}
0 \leq \frac{\partial y_n}{\partial x}(x_0,t_0) \leq  \sigma_{m,h}'(x_0) \leq \sup_{[-2,2]}  \sigma_{m,0}' \qedhere
\end{equation*}
\end{proof}
Next we shall show a uniform gradient bound over $[-k,k]$ for a short amount of time.
\begin{lem}[Local gradient bounds for a short time]\label{st grad bd}
Fix $k>0$. Then there exists $\tau > 0$ and $M_1(k) >0$ such that
\begin{equation*}
0 \leq \frac{\partial y_n}{\partial x}(x,t) \leq M_1, \quad \forall x \in [0,k], \ \forall t \in [0,\tau], \ \forall n \in \mathbb{N}.
\end{equation*}
In particular
\begin{equation*}
\abs{y_n(\cdot,t)}_{C^1([-k,k])} \leq M_1, \quad \forall t \in [0,\tau], \quad \forall n \in \mathbb{N}.
\end{equation*}
\end{lem}
The last line of this lemma again rephrases it as the special case of theorem \ref{loc grad bd} for for $T \leq \tau$. To prove lemma \ref{st grad bd} we foliate $[-k,k]\times\mathbb{R}$ for a short amount of time. We will not use geodesics to foliate, as their gradient becomes too shallow far out. Instead we use the same reasoning as in section \ref{chapter2} to construct an immortal, smooth solution to $\mathcal{V}=0$ from a suitable initial condition, and use vertical translations of this solution as a foliation. To be more precise, on the slightly larger domain $(0,k+1)$ consider the Dirichlet problem
\begin{equation*}
\begin{cases}
\mathcal{V}(F) =0 \quad &\text{in} \ (0,k+1) \times (0,\infty) \\
F(x,t) = 4x \quad &\text{on} \ \left([0,k+1] \times \{0\}\right) \cup \left(\{0,k+1\}\times[0,\infty)\right).  \\
\end{cases}
\end{equation*}
By the same reasoning as in section \ref{chapter2}, we have a solution $F:[0,k+1] \times [0,\infty) \rightarrow \mathbb{R}$ such that
\begin{itemize}
\item $\mathcal{V}(F) = 0$ on $(0,k+1) \times (0,\infty)$.
\item $F(x,0) = 4x$, $F(0,t)=0$ and $F(k+1,t) = 4(k+1)$.
\item $F(\cdot,t)$ is increasing for each $t > 0$.
\end{itemize}
By the vertical translation invariance of (\ref{CSF,2}), for each time $t > 0$, we can foliate with the solutions
\begin{equation*}
\mathcal{F}(t) := \{ F(\cdot, t) - h : h \in \mathbb{R} \}.
\end{equation*}
We are concerned with which curves in our foliation $\mathcal{F}(t)$ intersect with our solution $y_n(\cdot,t)$. Define
\begin{equation*}
H(x,t):= \Ima \left( F(\cdot,t) - y_n(\cdot,t) \vert_{[0,x]} \right), \quad \forall x \in [0,k+1],
\end{equation*}
so that a curve $F(\cdot,t) - h \in \mathcal{F}(t)$ intersects $y_n(\cdot,t)$ over $[0,x]$ if and only if $h \in H(x,t)$. The following proposition bounds the size of $H(x,t)$.
\begin{prop} For each $k>0$, there exists a constant $A_k>0$ such that
\begin{equation}\label{H}
H(x,t) \subseteq  \left[0,  4xe^{A_kt}\right], \quad \forall (x,t)\in [0,k+1]\times(0,\infty). 
\end{equation}
\end{prop}
\begin{proof}
 Fix $A_k>0$ to be determined later. We showed in section \ref{chapter2} that $\mathcal{V}(4x) \leq 0$. By a similar calculation
\begin{equation*}
\mathcal{V}(4xe^{A_kt}) =  4e^{A_kt}\left( A_k x - \phi'(x)(1+\mu(4e^{A_kt})) \right) \geq 4e^{A_kt}\left( A_k x - 2\phi'(x) \right).
\end{equation*}
Note,  for $x \in [0,1]$, $\phi'(x)=0$ so $A_k x - 2\phi'(x) \geq 0$, and for $x\in [1,k+1]$, if we choose $A_k := 2(k+1)^2 = 2\phi'(k+1)$, then $A_k x - 2\phi'(x) \geq A_k - 2\phi'(x) \geq 0$. So $ \mathcal{V}(4x e^{A_kt} ) \geq 0$ over $[0,k+1]$ and by the avoidance principle \ref{avoidance princ}
\begin{equation*}
4x  \leq F(x,t) \leq  4x e^{A_k t}, \quad \forall (x,t)\in [0,k+1]\times(0,\infty).
\end{equation*}
So for any $t \in [0,T]$, over $[0,x]$ we have
\begin{equation*}
- 1 \leq F(\cdot,t) - y_n(\cdot,t) \leq 4xe^{A_k t}
\end{equation*}
To finish the lemma, we note that for $h<0$
\begin{equation*}
F(x,0) - h > 4x \geq  y_n(x,0), \quad F(0,t) - h > 0 = y_n(0,t),\quad F(k+1,t) - h > 1 = y_n(k+1,t). 
\end{equation*}
So by the avoidance principle \ref{avoidance princ}, $h \notin H(x,t)$.
\end{proof}

\begin{proof}[Proof of lemma \ref{st grad bd}]
Choose $\tau := \frac{1}{A_k} \log ( 1 + \frac{1}{4k}) > 0$ so that equation (\ref{H}) becomes
\begin{equation}
H(k,t) \subseteq [0,  4k + 1], \quad \forall t \in [0,\tau].
\end{equation}
Fix $(x_0,t_0) \in [0,k] \times [0,\tau]$. Since $\mathcal{F}(t_0)$ is a foliation, $ \exists \ h_0 \in H(k,t_0)$ such that
\begin{equation*}
F(x_0,t_0) - h_0 = y_n(x_0,t_0).
\end{equation*}
Since $\frac{\partial F}{\partial x}(x,0) = 4 > \frac{\partial y_n}{\partial x}(x,0)$, the curves $F-h_0$ and $y_n$ intersect at a single point at time $t=0$. By equation (\ref{H}), for any $t \in [0,\tau]$
\begin{align*}
F(0,t) - h_0 &= -h_0 \leq 0 = y_n(0,t),\\
F(k+1,t) - h_0 &= 4(k+1) - h_0 \geq 3 > y_n(k+1,t).
\end{align*}
Therefore, at time $t_0$ the curves intersect only at $x_0$ by the intersection principle \ref{int princ}, giving
\begin{equation*}
0 \leq \frac{\partial y_n}{\partial x}(x_0,t_0) \leq  \frac{\partial F}{\partial x} (x_0,t_0) \leq \sup_{[0,k] \times [0,\tau]} \left( \frac{\partial F}{\partial x} \right) \qedhere
\end{equation*}
\end{proof}
We are now ready to prove theorem \ref{loc grad bd}. Here we shall use the family of curves $g_c$ that we constructed in Chapter \ref{chapter2} to foliate our space. Although this foliation doesn't cover all of space-time, the regions it misses are covered by lemmas \ref{ssn grad bd} and \ref{st grad bd}.

\begin{proof}[Proof of theorem \ref{loc grad bd}.]
Take $a>1$ and $\tau>0$ as in lemmas \ref{st grad bd} and \ref{ssn grad bd}. To prove the theorem, it suffices to show that for any $(x^*,t^*) \in [a,k] \times [\tau,T]$, we can find a gradient bound for $y_n$ at $(x^*,t^*)$.\par

We begin by switching gage for our solutions. View the vertical graphs $y_n$ as the horizontal graphs $x_n : [0,1] \times (0,\infty) \rightarrow [0,n]$. By the intermediate value theorem and the monotonicity of $x_n(\cdot,t^*)$, $\exists \ y^* \in (0,1)$ such that $x_n(y^*,t^*) = x^*$. Choose $\tilde{\tau}>0$ sufficiently small so that it is both less than $\tau$ and so that $g_{k+1}(\tilde{\tau}) \geq c_{k+1} (\tilde{\tau}) \geq k$. This implies that the region $\{ (x,y) \in [a,k] \times [0,1] \}$ is folliated by the curves
\begin{equation*}
\mathcal{G} := \{ (g_c(y,\tilde{\tau}),y) : y \in [0,1], \ c \in [a,k+1] \}.
\end{equation*}
In particular, there exists $c^* \in [a,k+1]$ such that
\begin{equation*}
g_{c^*}(y^*,\tilde{\tau}) = x^* = x_n(y^* , t^*).
\end{equation*}
Consider the intersection number of the curves $g_{c^*}(\cdot,t)$ and $x_n(\cdot, t + (t^* - \tilde{\tau}))$. As $y_n(0,\cdot) = 0$ and $y_n(n,\cdot)=1$, we have
\begin{equation*}
 x_n(0,t) = 0, \quad x_n(1,t) = n,  \quad \forall t>0.
\end{equation*}
In particular, as $g_{c^*}(\cdot,0) = c^* \in (0,n)$ and $x_n(\cdot, t^* - \tilde{\tau})$ is strictly increasing, the curves initially intersect once. Moreover, for any $t \geq 0$
\begin{align*}
g_{c^*}(0,t) &= c^*(t) > 0  = x_n(0,t + (t^* - \tilde{\tau})),\\
g_{c^*}(1,t) &= c^* < n  = x_n(1,t + (t^* - \tilde{\tau})),
\end{align*}
and by the intersection principle \ref{int princ}, the curves always intersect only once. Therefore
\begin{equation*}
\frac{\partial x_n}{\partial y} (y^*,t^*) \geq \frac{\partial g_{c^*}}{\partial y} (y^*,\tilde{\tau}).
\end{equation*}
By the smooth dependence on initial conditions for solutions to the Dirichlet problems $H_c(t)$, the map
\begin{equation*}
G:[0,1] \times (1,\infty) \times (0,\infty) \rightarrow (0,\infty), \quad G(y,c,t) :=  \frac{\partial g_c}{\partial y} (y,t),
\end{equation*}
is continuous. Hence for any $X \Subset (1,\infty) \times (0,\infty)$, there exists $\epsilon(X)>0$ such that
\begin{equation}\label{cor grad eq}
G([0,1] \times X) \geq  \epsilon > 0.
\end{equation}
In particular, choosing $X := [a,k+1]\times[\tau,T]$ in equation (\ref{cor grad eq}), there exists $\epsilon>0$ such that
\begin{equation*}
\frac{\partial x_n}{\partial y} (y^*,t^*) \geq \frac{\partial g_{c^*}}{\partial y} (y^*,\tilde{\tau}) = G(y^*,c^*,\tilde{\tau}) \geq \epsilon.
\end{equation*}
That is
\begin{equation*}
0 \leq \frac{\partial y_n}{\partial x} (x^*,t^*) \leq \frac{1}{\epsilon}. \qedhere
\end{equation*}
\end{proof}
\subsection{Higher order bounds}
In order to get higher order bounds on our sequence, we combine the local gradient bounds with interior Schauder estimates.
\begin{thm}[Local bounds]\label{loc ck bd}
Fix $j \in \mathbb{N}$ and $K \Subset \mathbb{R} \times (0,\infty)$. Then there exists $M_j >0$ such that
\begin{equation*}
\abs{y_n}_{P^j(K)} \leq M_j, \quad \forall n \in \mathbb{N}.
\end{equation*}
\end{thm}
\begin{proof}
There exists some $\epsilon > 0$ such that $K_{2\epsilon} \subseteq \mathbb{R} \times (0,\infty)$, where $K_{2\epsilon}$ denotes the $2 \epsilon$-fattening of $K$. We begin by substituting $y_n$ into the coefficients of $\mathcal{V}$. By theorem \ref{loc grad bd}, our operator $\mathcal{L}_n$ is then a strictly parabolic linear operator. Moreover, on $K_{2\epsilon}$, these coefficients are uniformly bounded in $L^{\infty}(K_{2\epsilon})$. Thus, we can apply De Giorgi-Nash-Moser \ref{DGNM} to conclude that our sequence of solutions $y_n$ are uniformly bounded in $P^{0,\alpha}(K_\epsilon)$ for some $\alpha \in (0,1)$. Using interior Schauder estimates \ref{L Schauder} the result follows.
\end{proof}

As a consequence of theorem \ref{loc ck bd}, we now have local uniform $P^j$-bounds for our sequence $y_n$, and by Arzela-Ascoli, $y_n$ converges to $y$ in $C^\infty_{\text{loc}}(\mathbb{R}\times(0,\infty))$. Hence $y$ is smooth with $\mathcal{V}(y) = \lim_{n\rightarrow \infty} \mathcal{V}(y_n) = 0$ on $\mathbb{R} \times(0,\infty)$.

\section{Long term behaviour}\label{chapter4}
We currently have a continuous function $y: \mathbb{R} \times [0,\infty) \rightarrow [-1,1]$ such that
\begin{enumerate}
\item [(i)] $y(\cdot , 0) \equiv 0$  on $\mathbb{R}$.
\item[(ii)] $y(\cdot,t)$ is an increasing odd function $\forall t \in (0,\infty)$.
\item [(iii)]  $y$ is smooth and satisfies $\mathcal{V}(y)=0$ on $\mathbb{R} \times (0,\infty)$.
\end{enumerate}
The final step in the proof of theorem \ref{main thm} is to show that our solution does not remain equal to zero as we flow forwards in time. To do this we construct a suitable barrier. In particular, we find a graphical solution that acts as a barrier to our sequence of solutions $y_n$.

\subsection{Finding a suitable barrier}

Let $\zeta : (0,\infty) \rightarrow  (1,\infty)$ be the solution to the ODE
\begin{equation}
    \frac{\partial}{\partial t}\zeta(t) = -\phi'(\zeta(t)),
\end{equation}
such that $\zeta(t) \rightarrow \infty$ as $t \searrow 0$. As discussed earlier in \ref{chapter choosing metric}, $\zeta(t) = t^{-1}$ for small $t>0$. Consider the barrier function $b:(0,\infty)\times (0,\infty) \rightarrow (0,\infty)$ given by $b(y,t) := t + \zeta(t) + \frac{1}{\log(1+y)}$. We shall show that as a horizontal graph, this is a supersolution to equation (\ref{CSF,1}). Since $\phi(b(y,t))>0$, we have the upperbound
\begin{equation*}
\frac{b_{yy}}{b_y^2 + e^{2\phi(b)}} = \frac{2\log(1+y) + \log(1+y)^2}{1 + e^{2\phi(b)} (1+y)^2} \leq 1.
\end{equation*}
Moreover, since $\phi'(x)$ is increasing, we have that $\phi'(b(y,t)) + \dot{\zeta}(t) = \phi'(b(y,t)) - \phi'(\zeta(t)) \geq 0$. Using the above inequalities and substituting $b$ into $\mathcal{H}$ gives
\begin{equation*}
\mathcal{H}(b) = 1 + \dot{\zeta}(t) - \frac{b_{yy}}{b_y^2 + e^{2\phi}} + \phi'(b(y,t)) \left( 1 + \frac{b_y^2}{b_y^2 + e^{2\phi}} \right) \geq 0.
\end{equation*}
So $b$ is a supersolution to (\ref{CSF,1}). Switching gage, $(x,t) \mapsto \exp \left( \frac{1}{x - (t+ \zeta(t))}\right) - 1$ is a supersolution to (\ref{CSF,2}) in the region $U:= \{ (x,t) \in (t + \zeta(t),\infty) \times (0,\infty)\}$. In particular,  using $\mathcal{V}(-y) = -\mathcal{V}(y)$ and the vertical translation invariance of $\mathcal{V}$, we have that the graph of the function $u(x,t) := 2-\exp \left( \frac{1}{x - (t+ \zeta(t))} \right)$ is a subsolution to (\ref{CSF,2}) in the region $U$. We can then modify $u$ to get a subsolution $\overline{u}$ (in the viscosity sense) defined on all of $(0,\infty) \times (0,\infty)$ by setting
\begin{equation*}
\overline{u}(x,t) := \begin{cases}
-1 &: (x,t) \notin B\\
\max \{ -1 , u(x,t) \} &: (x,t) \in B
\end{cases}
\end{equation*}
This $\overline{u}$ is our barrier. The following lemma shows that this barrier does indeed push our solution $y$ away from zero.
\begin{lem}
\begin{equation*}
\overline{u}(x,t) \leq y(x,t), \quad \forall (x,t) \in \mathbb{R} \times [0,\infty).
\end{equation*}
\end{lem}
\begin{proof}
Fix $n \in \mathbb{N}$.  On the parabolic boundary of the region where $y_n$ is defined
\begin{equation*}
\overline{u}(\cdot,0) = -1 \leq y_n(\cdot,0), \quad \overline{u}(-n,\cdot) = -1 = y_n(0,\cdot), \quad \overline{u}(n,\cdot) < 1 = y_n(n,\cdot).
\end{equation*}
By the maximum principle
\begin{equation*}
\overline{u}(x,t) \leq y_n(x,t), \quad \forall (x,t) \in [-n,n] \times [0,\infty),  \quad \forall n \in \mathbb{N}.
\end{equation*}
The result follows from the convergence of $y_n$ to $y$.
\end{proof}
For any $t>0$ and $\epsilon \in (0,1) $, setting $x_0:= t+\zeta(t) + \frac{1}{\log(1+\epsilon)}$, we have
\begin{equation*}
1-\epsilon \leq \overline{u}(x,t) \leq y(x,t), \quad \forall x > x_0.
\end{equation*}
This concludes the proof of theorem \ref{main thm}.

\section{Uniqueness of radial geodesics}\label{chapter5}
For the final section, we consider metrics of the form $g = dr^2 + e^{2\phi(r)} d\theta^2$ as in (\ref{eqn polar metric}) which are complete smooth $O(2)$-invariant metrics on the plane with non-positive curvature. As any such $g$ is complete, smooth and $O(2)$-invariant, we have the previous analytic definition of $g$ blooming at infinity (see definition \ref{defn bloom}). Moreover, under the additional assumption that the curvature is non-positive, we can show that an equivalent geometric formulation for $g$ blooming at infinity is that all closed solutions to CSF become extinct within a finite uniform time. We shall first make this statement precise, before using it to prove theorem \ref{thm2}.\par

\subsection{Geometric formulation of blooming at infinity}
Given a region $U \subseteq M$ within our surface, we want to quantify the maximal existence time for closed solutions to CSF which initially lie within $U$.
\begin{defn}\label{defn existence time}
For any subset $U \subseteq M$, let $\mathcal{C}(U)$ denote the class of smooth closed solutions to equation (\ref{eqn:tang})
\begin{equation*}
    \eta : S^1 \times [0,T) \rightarrow M,
\end{equation*}
such that $\eta(\cdot,0)\subseteq U$. For such a solution $\eta$, we say that its existence time is $T$. Define the existence time of the subset $U$ to be the supremum of all such existence times
\begin{equation*}
    \tau(U) = \sup \{ T \in (0,\infty) : T \textnormal{ is the existence time for some } \eta \in \mathcal{C}(U) \}.
\end{equation*}
\end{defn}
\begin{lem}[Equivalent formulations for blooming at infinity]\label{bloom defn}
Let $g = dr^2 + e^{2\phi(r)} d\theta^2$ as in (\ref{eqn polar metric}) be a complete smooth $O(2)$-invariant metric on the plane with non-positive curvature. For each $m \in \mathbb{N}$, let $S_m := \{ (r,\theta) : r \geq 0, \ \theta \in [0,\frac{2\pi}{m}] \} \subset \mathbb{R}^2$. The following conditions are equivalent.
\begin{enumerate}
\item $g$ allows blooming at infinity. That is, there exists $T \in (0,\infty)$ and a solution $R: (0,T) \rightarrow (0,\infty)$ to the ODE (\ref{eqn symm circle}) such that $R(t) \rightarrow \infty$ as $t \searrow 0$.

\item For any $t>0$, there exists $R(t) \in (0,\infty)$ such that, for any $m \in \mathbb{N}$ and $\eta \in \mathcal{C}(S_m)$, we have that $\eta(\cdot,t) \subseteq \overline{B_{R(t)}}$.

\item For any $m \in \mathbb{N}$, the existence time $\tau(S_m) < \infty$ (see definition \ref{defn existence time}).
\end{enumerate}
\end{lem}

\begin{proof}
We shall show 1$\implies$2$\implies$3$\implies$1.
\begin{itemize}
    \item [(1$\implies$2)] Let $\eta \in \mathcal{C}(S_m)$. By compactness, for any $\epsilon>0$, there exists some $\delta \in (0,\epsilon)$ such that $\eta(\cdot,\epsilon) \subseteq B_{R(\delta)}$. By the avoidance principle for closed curves, $\eta(\cdot,t) \subseteq B_{R(t - \epsilon + \delta)}$. Letting $\epsilon \searrow 0$ yields the result.
    \item [(2$\implies$3)] Fix $\eta \in \mathcal{C}(S_m)$. Either the existence time of $\eta$ is less than $1$, or by our assumption, there exists some $R>0$ independent of $\eta$ such that $\eta(\cdot,1) \subseteq \overline{B_R}$. Using that the curvature is non-positive and Gauss-Bonnet, we have that $T \leq 1 + \frac{| B_{R}|}{2\pi m}$ and hence $\tau(S_m) \leq 1 + \frac{| B_{R}|}{2\pi m} < \infty$.
    \item [(3$\implies$1)] Fix $r>0$ and consider the region $B_{r} \cap S_m$. We can flow the boundary of this region under CSF to get a solution $\eta : S^1 \times [0,T) \rightarrow S_m$ with existence time $T \leq \tau(S_m)$. Also consider the maximal solution $r(t):(0,T_0) \rightarrow (0,\infty)$ to the ODE (\ref{eqn symm circle}), starting from $r(0)=r$. We note that under the usual $O(2)$-action on the plane, the rotated curves $(\frac{2\pi j}{m} \cdot \eta)$ for $j \in \{0,1,\ldots,m-1\}$ are disjoint, and completely fill the region $B_{r}$. Therefore by lemma \ref{GB lem}
    \begin{equation*}
        T_0 \leq m \cdot T \leq m \cdot \tau(S_m).
    \end{equation*}
    In particular, taking $r \nearrow \infty$ gives     $\tau(\mathbb{R}^2) \leq m \cdot \tau(S_m) < \infty$. Now consider the sequence of maximal solutions $R_n : [-T_n,0) \rightarrow (0,\infty)$ to the ODE (\ref{eqn symm circle}) with $R_n(-T_n) = n$, for all $n \in \mathbb{N}$. Note that the $T_n$ are strictly increasing and bounded above by $\tau(\mathbb{R}^2)$, so they converge to some finite limit $T$. Taking the limit of the sequence $R_n$ in $n$ and reparameterising gives a solution $R : (0,T) \rightarrow (0,\infty)$ to the ODE (\ref{eqn symm circle}), with $R(t) \rightarrow \infty$ as $t \searrow 0$. \qedhere
\end{itemize}
\end{proof}
\subsection{Proof of theorem \ref{thm2}}

\uniq*
\begin{proof}[Proof of theorem \ref{thm2}]
Let $\gamma : \mathbb{R} \times [0,T] \rightarrow \mathbb{R}^2$ be any uniformly proper solution to CSF starting from the $x$-axis. Fix $m \in \mathbb{N}$ and $r>0$. Using lemma \ref{bloom defn} there exists $\eta_r \in  \mathcal{C}(S_m)$ with existence time greater than $T$ and a point $x \in S^1$ such that $\eta_r(x,T)$ lies outside the ball $B_{r}$ centred at the origin radius $r$. 
Consider now the rotated slice $ (\pi - \frac{2\pi}{m})\cdot{S_m} = \{ (r,\theta) : r \geq 0, \ \theta \in [\pi - \frac{2\pi}{m}, \pi] \}$, and the convex region $\Omega_m := S_m \cup (\pi - \frac{2\pi}{m})\cdot{S_m} \cup B_{\frac{1}{m}}$. We currently have a smooth closed curve $\eta_r(\cdot,0) \subseteq S_m$. We choose $\widehat{\eta}_r \in \mathcal{C}(\Omega_m)$ such that $\widehat{\eta}_r(\cdot,0)$ is a smooth closed curve in $\Omega_m$ enclosing both $\eta_r(\cdot,0)$ and its rotated image $(\pi - \frac{2\pi}{m})\cdot \eta_r(\cdot,0) \subseteq (\pi - \frac{2\pi}{m})\cdot{S_m}$. By the avoidance principle for closed curves, the existence time $\widehat{T}$ of $\widehat{\eta}_r$ is greater than $T$ and there exists points $x,y \in S^1$ such that both $\widehat{\eta}_r(x,T)$ and $\widehat{\eta}_r(y,T)$ lie outside of $B_r$, but with $\widehat{\eta}_r(x,T) \in S_m$ and $\widehat{\eta}_r(y,T) \in (\pi - \frac{2\pi}{m})\cdot{S_m}$.\par
For each $r > 0$, we  now apply the avoidance principle \ref{avoid} to the closed curve $\widehat{\eta}_r$ and the uniformly proper solution $\gamma$, as well as to the rotated closed curve $\pi \cdot \widehat{\eta}_r \in \mathcal{C}(\pi \cdot \Omega_m)$ and $\gamma$ to deduce that
\begin{equation*}
   \Ima{\gamma(\cdot,t)} \subseteq \Omega_m \cup \pi \cdot \Omega_m, \quad \forall m \in \mathbb{N}, \ \forall t \in [0,T].
\end{equation*}
Taking $m \rightarrow \infty$ gives
\begin{equation*}
  \Ima{\gamma(\cdot,t)} \subseteq  \bigcap_{m \in \mathbb{N}} \left(\Omega_m \cup \pi \cdot \Omega_m \right) = \{ (r,\theta) : r \geq 0, \theta \in \{0,\pi\} \}, \quad \forall t \in [0,T].
\end{equation*}
We have shown that $\Ima{\gamma(\cdot,t)}$ is the $x$-axis for each $t \in [0,T]$.
\end{proof}
\begin{appendix}
\section{Linear operators}
For $\Omega := (a,b) \Subset \mathbb{R}$, recall the notation $\Omega_T:= \Omega \times (0,T)$, $\Gamma_{T}:= \left(\Omega \times \{0\}\right) \cup \left(\partial \Omega \times [0,T) \right)$. Consider the linear operator
\begin{equation}\label{linear parabolic operator}
\mathcal{L}(u) := u_t - A(x,t)u_{xx} + B(x,t)u_x + C(x,t)u,
\end{equation}
where $A,B,C$ are bounded functions on $\Omega_T$, with $A(x,t) > 0$, and $C \geq -C_0$. The following maximum principle is taken from \cite[][Chaper II, Theorem 2.1]{ladyzhenskaia1988linear}
\begin{thm}[Maximum principle \cite{ladyzhenskaia1988linear}]\label{linear max princ}
Fix $\alpha \in (0,1]$ and $(x^*,t^*) \in \Omega_T$. If $u \in P^{2,\alpha}(\Omega_T)$ satisfies $\mathcal{L}u \leq 0$ on $\Omega_T$, then
\begin{equation*}
u(x^*,t^*) \leq \max \{ 0, \sup_{\Gamma_{t^*}}(u e^{C_0(t^*-t)} ) \}.
\end{equation*}
Alternatively if $u \in P^{2,\alpha}(\Omega_T)$ satisfies $\mathcal{L}(u) \leq 0$ on $\Omega_T$, then
\begin{equation*}
u(x^*,t^*) \geq \min \{ 0, \inf_{\Gamma_{t^*}}(u e^{C_0(t^*-t)} ) \}.
\end{equation*}
\end{thm}
The following regularity theorem is taken from \cite[][Chapter III, Theorem 10.1]{ladyzhenskaia1988linear}
\begin{thm}[De Giorgi-Nash-Moser \cite{ladyzhenskaia1988linear}]\label{DGNM}
Let $u \in P^2(\Omega_T)$ be a solution of $\mathcal{L}u = 0$ on $\Omega_T$ such that the coefficients of $\mathcal{L}$ satisfy 
\begin{equation*}
    \norm{A}_{L^\infty(\Omega_T)}, \norm{A^{-1}}_{L^\infty(\Omega_T)}, \norm{B}_{L^\infty(\Omega_T)}, \norm{C}_{L^\infty(\Omega_T)} \lesssim 1.
\end{equation*}
Fix $K \Subset \Omega_T$. Then there exists
\begin{equation*}
    \alpha( \norm{A}_{L^\infty(\Omega_T)}, \norm{A^{-1}}_{L^\infty(\Omega_T)}) \in (0,1),
\end{equation*}
and a constant 
\begin{equation*}
    C(K,\Omega_T,\norm{u}_{L^\infty(\Omega_T)}, \norm{A}_{L^\infty(\Omega_T)}, \norm{A^{-1}}_{L^\infty(\Omega_T)},\norm{B}_{L^\infty(\Omega_T)}, \norm{C}_{L^\infty(\Omega_T)}) > 0,
\end{equation*}
such that
\begin{equation*}
    \norm{u}_{P^{0,\alpha}(K)} \leq C.
\end{equation*}
\end{thm}

Given $u_0 \in C^{2,\alpha}(\Omega)$, $\psi \in P^{2,\alpha}(\Omega_T)$ and $f \in P^{0,\alpha}(\Omega_T)$, consider the Dirichlet problem
\begin{equation}\label{lin dir}
\begin{cases}
\mathcal{L}(u) = f \quad &\text{in} \ \Omega_{T}\\
u = u_0 \quad &\text{on} \ \overline{\Omega} \times \{ 0 \}\\
u = \psi \quad &\text{on} \ \{ a,b\} \times [0,T]
\end{cases}
\end{equation}
The following existence theorem is taken from \cite[][Chaper IV, Theorem 5.2]{ladyzhenskaia1988linear}.
\begin{thm}[Global Schauder estimate \cite{ladyzhenskaia1988linear}]\label{G Schauder}
Fix $\alpha \in (0,1)$ and $k \in \mathbb{N}_0$. Suppose $A,B,C, f \in P^{k,\alpha}(\Omega_T)$, $\psi \in P^{2+k , \alpha}(\Omega_T)$, and $u_0 \in C^{2+k,\alpha}(\Omega)$, with the auxiliary data satisfying the compatibility conditions of orders $0, 1, \ldots ,k$. Furthermore, assume $\exists \lambda > 0$ such that $A(x,t)\geq \lambda$ on $\Omega_T$, that is, $\mathcal{L}$ is uniformly parabolic. Then there exists a unique $u \in P^{2+k , \alpha}(\Omega_T)$ solving (\ref{lin dir}). Furthermore, there exists a constant $ C(\Omega,\lambda,k,\alpha)>0$ such that
\begin{equation*}
\abs{u}_{P^{2+k,\alpha}(\Omega_T)} \leq C \left( \abs{u_0}_{C^{2+k,\alpha}(\Omega)} + \abs{\psi}_{P^{k+2,\alpha}(\Omega_{T})} + \abs{f}_{P^{k,\alpha}(\Omega_{T})}  \right).
\end{equation*}
\end{thm}
The following regularity theorem is taken from \cite[][Chapter IV, Theorem 10.1]{ladyzhenskaia1988linear}.
\begin{thm}[Interior Schauder estimate \cite{ladyzhenskaia1988linear}]\label{L Schauder}
Suppose that the conditions from Theorem \ref{G Schauder} hold with $u \in P^{2+k , \alpha}(\Omega_T)$ solving (\ref{lin dir}). Suppose $K \Subset \Omega_T$. Then there exists a constant $ C(\Omega,\lambda,k,\alpha,K)>0$ such that
\begin{equation*}
\abs{u}_{P^{2+k,\alpha}(K)} \leq C \left( \abs{u}_{P^{k,\alpha}(\Omega_{T})} + \abs{f}_{P^{k,\alpha}(\Omega_{T})}  \right).
\end{equation*}
\end{thm}

Given $\alpha \in (0,1)$ and $u \in P^{2,\alpha}(\Omega_T)$, we define the zero set
\begin{equation*}
   Z:= \{ (x,t) \in \overline{\Omega} \times [0,T) : u(x,t)=0 \},
\end{equation*} 
and the zero set at time $t$
\begin{equation*}
    Z_t:= \{ x \in \overline{\Omega} : (x,t) \in Z \}.
\end{equation*}
For such a zero $(x,t) \in Z$, we say it is a simple zero if $u_x(x,t) \neq 0$ and a repeated zero if $u_x(x,t) = 0$. We also define a function which counts the number of zeros at each time $z:[0,T) \rightarrow \mathbb{N}_0 \cup \{ \infty \}$, $z(t) := \abs{Z_t}$.
Suppose $u$ is a solution to (\ref{lin dir}). Recall that $\psi$ is the auxiliary data on the parabolic walls. We need to impose conditions on $\psi$ to deduce the monotonicity of $z$. 
\begin{defn}
We say $\psi$ is \textit{nice} if, for each $x \in \{a,b\}$, either:
\begin{enumerate}
\item[(i)] $\psi(x,t) = 0 \quad \forall t \in [0,T]$.
\item [(ii)] $\psi(x,t) \neq 0 \quad \forall t \in [0,T]$.
\item[(iii)] $\psi(x,t) \neq 0 \quad \forall t \in (0,T]$, and $\psi(x,0)=0$ is a simple zero.
\end{enumerate}
\end{defn}
The following theorem is a monotonicity formula for the number of zeros of $u$. It is a slight modification of Angenent's original argument   \cite{angenent1988zero}. See \cite{PhDthesis} for a full proof.
\begin{thm}[Monotonicity of zeros]\label{sturm}
Suppose $\mathcal{L}$ is a strictly parabolic operator with smooth coefficients. Suppose $\alpha \in (0,1)$ and $u \in P^{2,\alpha}(\Omega_T)$ is a solution to (\ref{lin dir}). If $\psi$ is nice, then $z:[0,T) \rightarrow \mathbb{N}_0 \cup \{ \infty \}$ is decreasing and $z$ is finite at any positive time.
\end{thm}
\section{Quasi-linear parabolic operators}
Consider a quasi-linear parabolic operator
\begin{equation}\label{QL parabolic operator}
\mathcal{Q}(u) := u_t - A(x,u,u_x)u_{xx} + B(x,u,u_x),
\end{equation}
with $A(x,z,p),B(x,z,p) \in C^{\infty}(\overline{\Omega} \times \mathbb{R} \times \mathbb{R} )$ and $A(x,z,p) > 0$ on $\overline{\Omega} \times \mathbb{R} \times \mathbb{R}$.\par
Given a subsolution and a supersolution, their difference satisfies a linear parabolic inequality to which we can apply the maximum principle.
\begin{thm}[Avoidance principle]\label{avoidance princ}
Fix $\alpha \in (0,1]$ and $u_0,u_1 \in P^{2,\alpha}(\Omega_T)$ with $\mathcal{Q}u_0 \leq 0$ and $\mathcal{Q}u_1 \geq 0$ on $\Omega_T$. If $u_0 \leq u_1$ on $\Gamma_T$, then $u_0 \leq u_1$ on $\Omega_T$.
\end{thm}
\begin{proof}
Setting $v:=u_1 - u_0$ and $u(s) := su_1 + (1-s)u_0$
\begin{align*}
v_t &\geq A(x,u_1,(u_1)_x)(u_1)_{xx} - B(x,u_1,(u_1)_x) -A(x,u_0,(u_0)_x)(u_0)_{xx} + B(x,u_0,(u_0)_x)\\ 
&= \int_{0}^1 \frac{\partial}{\partial s} \left[  A(x,u(s), (u(s))_x) (u(s))_{xx}  - B(x,u(s),(u(s))_x) \right] ds \\
&= \left[ \int_{0}^1  A(x, u(s),(u(s))_x) ds \right] v_{xx} \\
&+ \left[\int_{0}^1 \left( \frac{\partial A}{\partial p}(x,u(s),(u(s))_x) (u(s))_{xx} -  \frac{\partial B}{\partial p}(x,u(s),(u(s))_x)  \right)  ds \right] v_x\\
&+ \left[\int_{0}^1 \left( \frac{\partial A}{\partial z}(x,u(s),(u(s))_x) (u(s))_{xx} -  \frac{\partial B}{\partial z}(x,u(s),(u(s))_x)  \right)  ds \right] v\\
& := \tilde{A}(x,t) v_{xx} - \tilde{B}(x,t) v_x - \tilde{C}(x,t) v
\end{align*}
In particular, define the linear parabolic operator
\begin{equation*}
\widetilde{\mathcal{L}}(u) := u_t - \tilde{A}(x,t) u_{xx} + \tilde{B}(x,t) u_x + \tilde{C}(x,t) u.
\end{equation*}
Note that $\tilde{A}, \tilde{B},\tilde{C}$ continuous on $\Omega_T$ and $\tilde{A}(x,t) = \int_{0}^1 A(x,u(s),u(s)_x) ds \geq 0$ on $\Omega_T$. By compactness there exists some lower bound $\tilde{C} \geq -C_0$ on $\Omega_T$. Since $\widetilde{\mathcal{L}}(v) \geq 0$ on $\Omega_T$, by \ref{linear max princ}
\begin{equation*}
v(x,t) \geq  \min \{ 0 , \inf_{\Gamma_T} \left(v e^{C_0(T - t)} \right) \} \geq  0, \quad \forall (x,t) \in \Omega_T. \qedhere
\end{equation*}
\end{proof}
Given $u_0 \in C^{2,\alpha}(\Omega)$, and $\psi \in P^{2,\alpha}(\Omega_T)$ satisfying the compatibility conditions of order $0$ ($\psi = u_0$ on $S$), we consider the Dirichlet problem
\begin{equation}\label{quasi dir}
\begin{cases}
\mathcal{Q}(u) = 0 \quad &\text{in} \ \Omega_{T}\\
u = u_0 \quad &\text{on} \ \overline{\Omega} \times \{ 0 \}\\
u = \psi \quad &\text{on} \ \{ a,b\} \times [0,T]
\end{cases}
\end{equation}
For each $s\in (0,T]$, we can restrict to the shorter time Dirichlet problem
\begin{equation*}
(D_s):= \begin{cases}
\mathcal{Q}(u) = 0 \quad &\text{in} \ \Omega_{s}\\
u = u_0 \quad &\text{on} \ \overline{\Omega} \times \{ 0 \}\\
u = \psi \quad &\text{on} \ \{ a,b\} \times [0,s]
\end{cases}
\end{equation*}
In one spatial dimension, we can always write our operator $\mathcal{Q}$ in divergence form
\begin{equation}\label{QLD}
\mathcal{Q}(u) := u_t - \frac{\partial}{\partial x} \left( a(x,u,u_x) \right) + b(x,u,u_x),
\end{equation}
with $A(x,z,p) = \frac{\partial a}{\partial p}(x,z,p)$ and $b(x,z,p) = B(x,z,p) + \frac{\partial a}{\partial x}(x,z,p) + \frac{\partial a}{\partial z}(x,z,p) \cdot p $. The following theorem is from \cite[][Chaper V, Theorem 6.1]{ladyzhenskaia1988linear}
\begin{thm}[Existence and uniqueness for strictly parabolic operators \cite{ladyzhenskaia1988linear}]\label{existence}
Fix $\alpha \in (0,1]$. Suppose that for each $M>0$ the coefficients of $\mathcal{Q}$ from (\ref{QL parabolic operator}) and (\ref{QLD}) satisfy
\begin{enumerate}
\item[(i)] $B(x,z,0) \geq 0, \quad   \forall (x,z) \in \overline{\Omega} \times \mathbb{R}$.
\item[(ii)] $\frac{\partial a}{\partial p} \lesssim 1$, $\abs{a} , \abs{\frac{\partial a}{\partial z}} \lesssim (1+\abs{p})$, $\abs{\frac{\partial a}{\partial x} }, \abs{b} \lesssim (1+\abs{p})^2, \quad \forall (x,z,p) \in \overline{\Omega} \times [-M,M] \times \mathbb{R}$.
\item[(iii)] $ 1 \lesssim \frac{\partial a}{\partial p}, \quad \forall (x,z,p) \in \overline{\Omega} \times [-M,M] \times \mathbb{R}$.
\end{enumerate}
Then there exists a unique solution $u \in P^{2,\alpha}(\Omega_T)$ to the Dirichlet problem $(D_T)$.
\end{thm}

Suppose now that $\mathcal{Q}$ satisifies criteria $(i)$ and $(ii)$, but does not satisfy $(iii)$ ($\mathcal{Q}$ is not strictly parabolic). 
\begin{thm}[Short-time existence and uniqueness]\label{loc solv}
Fix $\alpha \in (0,1]$. Suppose that for each $M>0$ the coefficients of $\mathcal{Q}$ from (\ref{QL parabolic operator}) and (\ref{QLD}) satisfy
\begin{enumerate}
\item[(i)] $B(x,z,0) \geq 0, \quad   \forall (x,z) \in \overline{\Omega} \times \mathbb{R}$.
\item[(ii)] $\frac{\partial a}{\partial p} \lesssim 1$, $\abs{a} , \abs{\frac{\partial a}{\partial z}} \lesssim (1+\abs{p})$,$\abs{\frac{\partial a}{\partial x} }, \abs{b} \lesssim (1+\abs{p})^2, \quad \forall (x,z,p) \in \overline{\Omega} \times [-M,M] \times \mathbb{R}$.
\end{enumerate}
Then, there exists $s \in (0,T]$ and a unique $u \in P^{2, \alpha}(\Omega_{s})$ such that $u$ solves the Dirichlet problem $(D_s)$.
\end{thm}
\begin{proof}
Let $M:= \abs{u_0}_{C^1(\Omega)} < \infty$ and let $\chi$ be any smooth bump function supported on $[-2,2]$ and equal to $1$ on $[-1,1]$. We define a new coefficient
\begin{equation*}
\tilde{a}(x,z,p) := a(x,z,0) + \int_0^p \chi(\frac{s}{M}) \frac{\partial a}{\partial p}(x,z,s) + (1 - \chi(\frac{s}{M})) \ ds,
\end{equation*}
a new quasi-linear operator
\begin{equation*}
\widetilde{\mathcal{Q}}(u) := u_t - \frac{\partial}{\partial x} \left( \tilde{a}(x,u,u_x) \right) + b(x,u,u_x),
\end{equation*}
and a class of new Dirichlet problems
\begin{equation*}
(\widetilde{D_s}):= \begin{cases}
\widetilde{\mathcal{Q}}(u) = 0 \quad &\text{in} \ \Omega_{s}\\
u = u_0 \quad &\text{on} \ \overline{\Omega} \times \{ 0 \}\\
u = \psi \quad &\text{on} \ \{ a,b\} \times [0,s]
\end{cases}
\end{equation*}
Observe the following:
\begin{itemize}
\item Since $\tilde{a} = a$ for $\abs{p} \leq M$, $\widetilde{Q}$ satisifies $(i)$.
\item $\frac{\partial \tilde{a}}{\partial x}(x,z,p) = \frac{\partial a}{\partial x}(x,z,0) + \int_{0}^p \chi(\frac{s}{M_1}) \frac{\partial^2 a}{\partial x \partial p}(x,z,s) ds$.
\item $\frac{\partial \tilde{a}}{\partial z}(x,z,p) = \frac{\partial a}{\partial z}(x,z,0) + \int_{0}^p \chi(\frac{s}{M}) \frac{\partial^2 a}{\partial z \partial p}(x,z,s) ds$.
\item $\frac{\partial \tilde{a}}{\partial p}(x,z,p) = \chi(\frac{p}{M}) \frac{\partial a}{\partial p}(x,z,p) + (1 - \chi(\frac{p}{M}))$.
\end{itemize}
From the above $\widetilde{\mathcal{Q}}$ satisfies $(ii)$. Finally, since $\frac{\partial \tilde{a}}{\partial p} \equiv 1$ outside of a compact set, $\widetilde{\mathcal{Q}}$ satisfies $(iii)$. By \ref{existence}, there exists $\tilde{u} \in P^{2,\alpha}(\Omega_T)$ solving $(\widetilde{D_T})$. Moreover, by the continuity of $\tilde{u}$ and $\tilde{u}_x$, there exists $s \in (0,T]$ such that
\begin{equation*}
\abs{\tilde{u}(\cdot,t)}_{C^1(\Omega)} \leq M, \quad \forall t \in [0,s].
\end{equation*}
In particular, since $\widetilde{\mathcal{Q}} = \mathcal{Q}$ on $\overline{\Omega} \times \mathbb{R} \times [-M,M]$, we have that $\tilde{u} \in P^{2,\alpha}(\Omega_s)$ solves $({D_s})$. Finally, if $u_1,u_2 \in P^{2,\alpha}(\Omega_s)$ are solutions to $(D_s)$, then by \ref{avoidance princ} we have $u_1 = u_2$ on $\Omega_s$.
\end{proof}

\begin{thm}[Existence of maximal solutions]\label{max sol}
Fix $\alpha \in (0,1]$. There exists a unique pair $\tau \in (0,T]$ and $u:\overline{\Omega} \times [0,\tau) \rightarrow \mathbb{R}$ such that
\begin{enumerate}
\item[(A)] $u \in P^{2,\alpha}(\Omega_s)$, $\forall s \in (0,\tau)$.
\item[(B)] $u$ solves $(D_s)$, $\forall s \in (0,\tau)$.
\item[(C)] If $\tau < T$ then $u \notin P^{2,\alpha}(\Omega_\tau)$ and $\limsup_{s \rightarrow \tau} \abs{u(\cdot,s)}_{C^1(\Omega)} = \infty$.
\end{enumerate}
\end{thm}
\begin{proof}
By \ref{loc solv}
\begin{equation*}
\tau:= \sup \{ s \in (0,T] : \exists u \in P^{2,\alpha}(\Omega_s) \text{ such that } u \text{ solves } (D_s) \},
\end{equation*}
is well defined. By uniqueness the solutions agree on overlaps, and give a well defined, unique function $u:\overline{\Omega} \times [0,\tau) \rightarrow \mathbb{R}$ satisfying properties $(A)$ and $(B)$. For $\tau < T$, assume that $u \in P^{2,\alpha}(\Omega_\tau)$. Then $u(\cdot,\tau) \in C^{2,\alpha}(\Omega)$ and we can reapply \ref{loc solv} to get a solution $\hat{u} \in P^{2,\alpha}(\Omega_{\tau,\tau+\epsilon})$ for some $\epsilon > 0$. By virtue of the PDE that they solve, $u$ and $\hat{u}$ piece together to give $u \in P^{2,\alpha}(\Omega_{\tau+\epsilon})$ solving $(D_{\tau+\epsilon})$, contradicting the definition of $\tau$. So $u \notin P^{2,\alpha}(\Omega_\tau)$.

Finally, assume that $\limsup_{s \rightarrow \tau} \abs{u(\cdot,s)}_{C^1(\Omega)} < \infty$ so that $\abs{u} ,\abs{u_x} \leq M$ for some positive constant $M > 0$. Consider the Dirichlet problem $(\widetilde{D_\tau})$ defined in the proof of \ref{max sol}. By \ref{existence} there exists a unique $\tilde{u} \in P^{2,\alpha}(\Omega_\tau)$ solving $(\widetilde{D_\tau})$. Since $\widetilde{\mathcal{Q}} = \mathcal{Q}$ on $\overline{\Omega} \times \mathbb{R} \times [-M_1,M_1]$, $u$ also solves $(\widetilde{D_s})$ for $s \in (0,\tau)$. Therefore, by the uniqueness of solutions, $\tilde{u}$ is an extension of $u$ in $P^{2,\alpha}(\Omega_\tau)$, contradicting what we have just previously shown.
\end{proof}
\begin{thm}[Regularity of solutions]\label{reg of sol}
Fix $\alpha \in (0,1]$. Suppose the auxiliary data $u_0 \in C^{\infty}(\Omega)$, $\psi \in C^{\infty}(\Omega_T)$ for (\ref{quasi dir}) satisfy the compatibility conditions of all orders. Let $u: \overline{\Omega} \times [0,\tau) \rightarrow \mathbb{R}$ be the maximal solution from theorem \ref{max sol}. Then for any $s \in (0,\tau)$, $u \in C^{\infty}(\Omega_s)$.
\end{thm}
\begin{proof}
Fix $s \in (0,\tau)$. Since $u \in P^{2,\alpha}(\Omega_s)$, substituting this $u$ into the coefficients of our operator gives a linear operator with coefficients in the class $P^{1,\alpha}(\Omega_s)$. Applying\ref{G Schauder} and bootstrapping gives $u \in \bigcap_{k \geq 1} P^{k,\alpha}(\Omega_s) = C^{\infty}(\Omega_s)$.
\end{proof}
Given two solutions $u_1,u_2 \in C^{\infty}(\Omega_T)$ of $\mathcal{Q}=0$, we define their intersection number at time $t$ to be
\begin{equation*}
I(t) := \abs{ \{ x \in \overline{\Omega} : u_1(x,t) = u_2(x,t) \}}
\end{equation*}
We now apply \ref{sturm} to the difference $u_1-u_2$.
\begin{thm}[Intersection principle]\label{int princ}
Let $u_1,u_2 \in C^{\infty}(\Omega_T)$ with $\mathcal{Q}(u_i) = 0$ on $\Omega_T$ and $u_1-u_2$ nice on $S_T$. Then $I(t)$ is decreasing for all time and is finite for positive time.
\end{thm}
\begin{proof}
Setting $v:= u_1 - u_2$ as in the proof of \ref{avoidance princ}, we see that $\widetilde{\mathcal{L}}(v) = 0$ on $\Omega_{T}$. Moreover, since $u_1,u_2$ are smooth, their gradients are bounded and there exists $M \geq 0$ such that $\abs{u(s)_x} \leq M$. By the extreme value theorem, $A$ attains its minimum $\lambda >0$ for $\abs{p} \leq M$, and hence
\begin{equation*}
\tilde{A}(x,t) = \int_{0}^1 A(x,u(s)_x) ds \geq \int_{0}^1 \lambda \ ds = \lambda.
\end{equation*}
This means $\widetilde{\mathcal{L}}$ is strictly parabolic on $\Omega_{T}$. The coefficients of $\widetilde{\mathcal{L}}$ are smooth and \ref{sturm} applies.
\end{proof}

\section{Miscellaneous}
We include a statement for the classical avoidance principle between a closed solution and a uniformly proper solution. Note that without the uniformly proper hypothesis on the non-closed solution, the theorem fails (consider the curve $\tilde{\gamma}$ from remark \ref{remark arm} and a closed solution $\eta$ disjoint from the $x$-axis but intersecting the line $L$).
\begin{lem}[Avoidance principle for CSF]\label{avoid}
Let $\gamma: \mathbb{R} \times [0,T] \rightarrow M$ be a uniformly proper solution to CSF and $\eta : S^1 \times [0,T] \rightarrow M$ a closed solution to CSF. If the curves are initially disjoint, then they remain disjoint.
\begin{equation*}
  \Ima(\eta(\cdot,0)) \cap \Ima(\gamma(\cdot,0)) = \emptyset  \ \implies \ \Ima(\eta(\cdot,t)) \cap \Ima(\gamma(\cdot,t)) = \emptyset, \ \forall t \in [0,T].
\end{equation*}
\end{lem}
\begin{proof}
It suffices to show that there is some positive first hitting time $t_0 \in (0,T]$. The result then follows from the usual maximum principle. Since $\eta$ is continuous, there exists some $x_0 \in M$ and $R>0$ such that Im$(\eta) \subseteq B_R(x_0)$. Since $\gamma$ is uniformly proper, there exists $K \Subset \mathbb{R}$ such that $\gamma^{-1}(B_R(x_0)) \subseteq K \times [0,T]$. Consider the distance function $d: S^1 \times K \times [0,T] \rightarrow [0,\infty)$,
\begin{equation*}
    d_t(x,y) := |\eta(x,t) - \gamma(y,t)|, \quad \forall (x,y,t) \in S^1 \times K \times [0,T].
\end{equation*}
Note that the distance between the two curves at any time $t \in [0,T]$ is given by
\begin{equation*}
    D(t) := \inf_{(x,y)\in S^1 \times K} d_t(x,y).
\end{equation*}
Since the solutions are initially disjoint, $D(0)>0$. By the compactness of $S^1 \times K$, $D(t)>0$ for sufficiently small $t$.
\end{proof}

The following lemma gives the explicit construction of the warping function $\phi$ whose existence we claimed in section \ref{chapter2}.
\warp
\begin{proof} 
Define the following bump functions
\begin{equation*}
f_1(x) := \begin{cases}
0 : x \leq 0\\
e^{-\frac{1}{x}} : x>0
\end{cases} \quad
f_2(x) := \frac{f_1(x)}{f_1(x) + f_1(\frac{1}{4}-x)} \quad f_{3}(x):= \frac{f_2(x-1) +  8 f_2\left(x - \frac{7}{4}\right)}{9}.
\end{equation*}
For all $x>0$, we then define 
\begin{equation*}
    \phi(x) := \int_0^x y^2 \cdot f_3(y) \ dy.
\end{equation*}
It is routine to check that $\phi$ is a smooth function with the necessary properties.
\end{proof}

The following lemma is a very simple observation about the how the gradients of two curves intersecting at a single point are ordered. Although basic, this lemma is used repeatedly in sections \ref{chapter2} and \ref{chapter3} in conjunction with foliation arguments.
\begin{lem}\label{fact}
\textnormal{If $u_1,u_2 \in C^1([\alpha,\beta])$ intersect at a single point $x_0 \in [\alpha,\beta]$, and either}
\begin{enumerate}
\item [$(A)$] \textnormal{$x_0 \in [\alpha,\beta)$ with $u_1(\beta) < u_2(\beta)$.}
\item [$(B)$] \textnormal{$x_0 \in (\alpha,\beta]$ with $u_1(\alpha) > u_2(\alpha)$.}
\end{enumerate}
\textnormal{Then $(u_1)'(x_0) \leq (u_2)'(x_0)$.}
\end{lem}
\begin{proof}
In case $(A)$, $u_1 \leq u_2$ on $[x_0,\beta]$ and so
\begin{equation*}
(u_1)'(x_0) = \lim_{h \rightarrow 0^+} \left( \frac{u_1(x_0 + h) - u_1(x_0)}{h} \right) \leq \lim_{h \rightarrow 0^+} \left( \frac{u_2(x_0 + h) - u_2(x_0)}{h} \right) = (u_2)'(x_0).
\end{equation*}
For $(B)$ use the left-sided limit instead.
\end{proof}

The following lemma provides an upperbound for the existence time of a closed solution to CSF inside a simply connected negatively curved space.

\begin{lem}\label{GB lem}
Let $(\mathbb{R}^2 , g)$ be a Hadamard surface ($g$ non-positive curvature) and $\eta_i : S^1 \times [0,T_i) \rightarrow \mathbb{R}^2$ be a family of closed disjoint solutions to CSF for $i \in \{1,\ldots,k\}$. Suppose $\eta : S^1 \times [0,T) \rightarrow \mathbb{R}^2$ is a maximal closed solution to CSF such that the region enclosed by $\eta(\cdot,0)$ contains all of the curves $\bigcup_{i=1}^k \eta_i(\cdot,0)$. Then $T \leq \frac{\alpha}{2\pi} + \sum_{i=1}^k T_i$, where $\alpha:= |\eta(\cdot,0)| -  \sum_{i=1}^k |\eta_i(\cdot,0)|$ is the initial area discrepancy.
\end{lem}
\begin{proof}
Let $\Gamma(t)$, $\Gamma_i(t)$ denote the regions enclosed by the curves $\eta(\cdot,t)$ and $\eta_i(\cdot,t)$ respectively. Without loss of generality $0 =: T_{k+1} < T_k \leq \cdots \leq T_1$. Then by the avoidance principle for closed curves, for each $m \in \{1,\ldots,k\}$ we have
\begin{equation*}
    \bigcup_{i=1}^m \Gamma_i(t) \subseteq \Gamma(t), \quad \forall t \in (T_{m+1} , T_m).
\end{equation*}
Let $A(t)$, $A_i(t)$ denote the areas of $\Gamma(t)$, $\Gamma_i(t)$ respectively, so that $A(0) = \alpha + \sum\limits_{i=1}^k A_i(0)$. For $t \in (T_{m+1},T_m)$ we apply Gauss-Bonnet to give
\begin{equation*}
    \partial_t A = -2\pi + \int_{\Gamma(t)} K dA \leq
    -2\pi + \sum_{i=1}^m \int_{\Gamma_i(t)} K dA = 2\pi (m-1) + \sum_{i=1}^m  \partial_t A_i.
\end{equation*}
Integrating, we have for each $m \in \{1,\ldots,k\}$
\begin{equation}\label{eqn appendix misc}
    A(T_m) - A(T_{m+1}) \leq 2\pi (m-1)(T_m - T_{m+1}) + \sum\limits_{i=1}^m A_i(T_m) - A_i(T_{m+1}).
\end{equation}
Summing (\ref{eqn appendix misc}) over $m \in \{1,\ldots,k\}$
\begin{equation*}
  A(T_1) \leq A(0) - \sum\limits_{i=1}^{k} A_i(0) +  2\pi \sum\limits_{i=1}^{k-1} T_i = \alpha + 2\pi \sum\limits_{i=1}^{k-1} T_i.
\end{equation*}
Applying Gauss-Bonnet once more
\begin{equation*}
 T \leq T_1 + \frac{A(T_1)}{2\pi} = \frac{\alpha}{2\pi} + \sum\limits_{i=1}^{k} T_i. \qedhere
\end{equation*}
\end{proof}

\end{appendix}

\printbibliography

\end{document}